\documentclass[11pt,twoside,notitlepage]{article}
\usepackage[scaled=0.92]{helvet}
\usepackage[T1]{fontenc}
\usepackage[utf8]{inputenc}
\usepackage[a4paper]{geometry}
\usepackage{xcolor}
\usepackage{textcomp}
\usepackage{amsmath}
\usepackage{amsthm}
\usepackage{amssymb}

\makeatletter

\allowdisplaybreaks
\usepackage{amsmath}
\usepackage{amssymb}
\usepackage{amsfonts}
\usepackage{amsthm}
\usepackage{enumerate}
\usepackage{graphicx}

\usepackage{color}

\usepackage[labelfont=bf,margin=1cm,justification=justified,labelsep=period]{caption}

\usepackage{amsthm}
\usepackage[nottoc,numbib]{tocbibind}

\theoremstyle{plain}
\newtheorem{theorem}{Theorem}[section]
\newtheorem{proposition}[theorem]{Proposition}

\newtheorem{lemma}[theorem]{Lemma}
\theoremstyle{definition}

\newtheorem{remark}[theorem]{Remark}
\newtheorem{hypothesis}[theorem]{Hypothesis}

\newcommand{\E}[1]{\mathbb{E} \left[ #1 \right]}

\newcommand{\R}{\mathbb{R}}

\newcommand{\C}{\mathcal{C}}
\newcommand{\F}{\mathcal{F}}
\newcommand{\hac}{\mathcal{H}}
\newcommand{\axi}{|\xi|}

\newcommand{\intt}{\int_0^t}

\newcommand{\hmin}{H_{\inf}}
\newcommand{\hmax}{H_{\sup}}
\newcommand{\beq}{\begin{equation}}
\newcommand{\eeq}{\end{equation}}

\begin{document}
\title{Convergence in law for quasi-linear SPDEs}
\date{\today}
\author{Maria Jolis\thanks{Departament de Matemàtiques, Universitat Autònoma de Barcelona, 08193
Bellaterra, Catalonia, Spain, E-mail: maria.jolis@uab.cat}, 
Salvador Ortiz-Latorre\thanks{Department of Mathematics, University of Oslo, P.O. Box 1053 Blindern,
	N-0316 Oslo, Norway, E-mail: salvadoo@math.uio.no} and 
Lluís Quer-Sardanyons\thanks{Departament de Matemàtiques, Universitat Autònoma de Barcelona, 08193
Bellaterra, Catalonia, Spain, E-mail: lluis.quer@uab.cat}}
\maketitle
\begin{abstract}
We consider the quasi-linear stochastic wave and heat equations in $\mathbb{R}^d$ with $d\in \{1,2,3\}$ and $d\geq 1$, respectively, and perturbed by an additive Gaussian noise which is white in time and has a homogeneous spatial correlation with spectral measure $\mu_n$. We allow the Fourier transform of $\mu_n$ to be a genuine distribution. Let $u^n$ be the mild solution to these equations. We provide sufficient conditions on the measures $\mu_n$ and the initial data to ensure that $u^n$ converges in law, in the space of continuous functions, to the solution of our equations driven by a noise with spectral measure $\mu$, where $\mu_n\to\mu$ in some sense. We apply our main result to various types of noises, such as the anisotropic fractional noise. We also show that we cover existing results in the literature, such as the case of Riesz kernels and the fractional noise with $d=1$. 
\end{abstract}

\medskip


\medskip


\section{Introduction}

We consider the stochastic wave equation
\begin{equation}
	\tag{SWE\ensuremath{_{n}}}\begin{cases}
		\dfrac{\partial^{2}u^{n}}{\partial t^{2}}(t,x)- \Delta u^n(t,x) = b(u^{n}(t,x))+\dot{W}^{n}(t,x),\\[0.2cm] 
		u^{n}(0,x)=u_{0}(x),\quad x\in\mathbb{R}^{d},\\[0.1cm] 
		u^{n}_{t}(0,x)=v_{0}(x),\quad x\in\mathbb{R}^{d},
	\end{cases}\label{eq: waven}
\end{equation}
defined in $(t,x)\in[0,\infty)\times\mathbb{R}^{d}$ with $d\in\left\{ 1,2,3\right\}$, and 
the stochastic heat equation
\begin{equation}
\tag{SHE\ensuremath{_{n}}}\begin{cases}
\dfrac{\partial u^{n}}{\partial t}(t,x)-\dfrac12 \Delta u^n(t,x) = b(u^{n}(t,x))+\dot{W}^{n}(t,x),\\[0.2cm]
u^{n}(0,x)=u_{0}(x),\quad x\in\mathbb{R}^{d},
\end{cases}\label{eq: heatn}
\end{equation}
defined in $(t,x)\in[0,\infty)\times\mathbb{R}^{d}$,  $d\geq 1$. 
The initial conditions $u_{0},v_{0}$ are deterministic functions
satisfying some assumptions which will be specified later on. The function $b$ is assumed to be globally Lipschitz. For
any $n\geq1$, the noise $\dot{W}^{n}$ is assumed to be white in
time and colored in space. We now give its detailed definition.

Let $(\Omega,\mathcal{A},\mathbb{P})$ be a complete probability space. On
the linear space $\mathcal{D}:=\mathcal{C}_{0}^{\infty}([0,\infty)\times\mathbb{R}^{d})$ of infinitely differentiable functions with compact support,
consider a spatially homogeneous Gaussian noise $\{W^{n}(\varphi),\ \varphi\in\mathcal{D}\}$,
namely a Gaussian stochastic process indexed on $\mathcal{D}$ such that
$\mathbb{E}\left[W^{n}(\varphi)\right]=0$, for all $\varphi\in\mathcal{D}$,
and with covariance structure 
\begin{equation}
\mathbb{E}\left[W^{n}(\varphi)W^{n}(\psi)\right]=\int_{0}^{\infty}\int_{\mathbb{R}^{d}}\mathcal{F}\varphi(t,\cdot)(\xi)\overline{\mathcal{F}\psi(t,\cdot)(\xi)}\mu_{n}(d\xi)dt,\quad\varphi,\psi\in\mathcal{D},
\label{eq:19}
\end{equation}
where $\mu_{n}$ is a non-negative tempered measure on $\mathcal{B}(\mathbb{R}^{d})$,
for all $n\geq1$. We refer to $\mu_{n}$ as the spectral
measure of the noise $\dot{W}^{n}$, and we recall that $\mu_{n}$
is necessarily symmetric (see \cite[Chap. VII, Théorème XVII]{Schwartz}). In \eqref{eq:19},
$\mathcal{F}$ denotes the Fourier transform
on $L^{1}(\mathbb{R}^{d})$, which is defined by 
$$\F f(\xi)=\int_{\R^d}e^{-i <\xi,x>}f(x)dx, \quad f\in L^1(\R^d),$$
where $<\xi, x>=\sum_{i=1}^{d}\xi_i x_i$ is the Euclidean inner product in $\R^d$. 
As usual, we introduce the Hilbert space $\hac_n$, which is the completion of $\mathcal{D}$ with respect to the inner product
\[
\langle \varphi,\psi\rangle_n:=\mathbb{E}\left[W^{n}(\varphi)W^{n}(\psi)\right], \quad 
\varphi,\psi \in \mathcal{D}.
\]
Then, the noise $W^n$ can be extended to a family of centered and Gaussian random variables 
$\{W^n(g),\, g\in \hac_n\}$ such that 
\[
\E{W^n(g_1)W^n(g_2)}=\langle g_1,g_2\rangle_n, \quad g_1,g_2\in\hac_n.
\]
For any $g\in \hac_n$, we say that the Gaussian random variable $W^n(g)$ is the Wiener integral of $g$ and we use the notation 
\[
\int_0^\infty \int_{\R^d} g(t,x)W^n(dt,dx):=W^n(g).
\]
All stochastic integrals appearing throughout the paper will be considered in this sense,
Owing to \cite[Lem. 3.2]{NQ}, one can deduce that any deterministic function $t\in\R_+\to g(t)$ with values in the space of distributions with rapid decrease and satisfying 
\[
\int_0^\infty\int_{\R^d} |\F g(t)(\xi)|^2 \mu_n(d\xi)dt<\infty,
\] 
belongs to $\hac_n$. Indeed, the hypotheses of \cite[Lem. 3.2]{NQ} require that $g$ is a non-negative distribution, but taking a close look at the proof one realizes that such a condition is not necessary here.  

We point out that we do not assume that the Fourier transform of the measure $\mu_n$ (in the sense of Schwartz distributions) is a function (or a measure). The latter case corresponds to the theory developed by Dalang in \cite{DalangEJP}. This is a key observation, because we aim to cover, at least, the case in which the spatial covariance structure is that of a fractional Brownian motion with Hurst index $H_n\in (0,1)$ (see Section \ref{sec: frac1}). This corresponds to the spectral measure
\beq
\mu_n(d\xi)=C_{H_n} |\xi|^{1-2H_n} d\xi, \quad \xi\in \R,
\label{eq:371}
\eeq
where $H_n\in (0,1)$ and the constant $C_{H_n}$ is given in \eqref{eq:14}. We note that Dalang's setting would only allow us to deal with the case $H_n\in[\frac12,1)$. If $H_n\in (0,\frac 12)$, we recall that the Fourier transform of $\mu_n$ is a genuine distribution (see \cite[Ch. 1, Sec. 3]{Gel}).

The aim of the paper is to provide sufficient conditions on the family of spectral measures $\mu_n$, $n\geq 1$, and the initial data, ensuring that the solution $u^n$ converges in law, in the space $\C([0,T]\times \R^d)$ of continuous functions, to the random field $u$ which solves the same kind of stochastic PDEs but driven by a Gaussian spatially homogeneous noise with spectral measure $\mu$, where $\mu_n\to\mu$ in some sense. On the space  $\C([0,T]\times \R^d)$, we consider the usual topology of uniform convergence in compact sets. We refer the reader to Theorem \ref{thm: main} for the precise statement of the the main result, and the assumptions on $\mu_n$ are given by Hypotheses {\bf{(H1)}} and {\bf{(H2)}} below. Indeed, as it will be made precise in Lemma \ref{lem: mu}, the measures $\mu_n$ and $\mu$ fulfill the following integrability conditions: there exists $q\in (0,2)$ such that
\[
\int_{\mathbb{R}^{d}}\frac{\mu_{n}(d\xi)}{1+|\xi|^{q}}<\infty \quad
\text{and}\quad \int_{\mathbb{R}^{d}}\frac{\mu(d\xi)}{1+|\xi|^{q}}<\infty.
\] 
We point out that these conditions are natural and consistent with the existing results on path continuity for stochastic PDEs (see, e.g., \cite{MS,Sanz-Sarra,SS}).

The main motivation for considering such a problem comes from the results obtained in \cite{GJQ-Bernoulli} (see also \cite{GJQ-SPA} for the linear multiplicative counterpart), where the authors consider $d=1$, $\mu_n$ is given by \eqref{eq:371} and $\mu=\mu_0$, with $H_n\to H_0$. In the present paper, we consider space dimensions greater than 1, and we make sure that the latter case is covered by our result, as well as other two important examples. Namely, the anisotropic fractional noise, which is tackled in Section \ref{sec: ani}, and the Riesz kernel (see Section \ref{sec: riesz}). In the latter example, the analogous problem of weak convergence has already been studied in \cite{Bezdek} and \cite{Tao} for the one-dimensional heat and wave equations, respectively. In the quasi-linear models that we are considering here, their results are particular cases of our Theorem \ref{thm: main}. However, we should mention that in \cite{Bezdek,Tao} the authors consider a general non-linear multiplicative noise. We have stuck to the quasi-linear form of equations \eqref{eq: waven} and \eqref{eq: heatn} because we aim at having a sufficiently general result which could cover {\it rough} noises in space. In this sense, we postpone the study of the corresponding linear multiplicative settings (Hyperbolic and Parabolic Anderson Models) for future work, since the needed techniques are completely different from the type of considerations that we are using in the present paper. We also tried to verify that our main assumptions are fulfilled for the isotropic fractional noise (see Section \ref{sec: iso}). Indeed, in the latter case, we deduce the form of the corresponding spectral measure (we could not find a reference where this was specified) and we show that it does not even satisfy Dalang's condition, unless $d=1$, which corresponds to the setting considered in \cite{GJQ-Bernoulli}. 
Finally, we also mention that continuity in law for the solution to one-dimensional stochastic PDEs driven by a time-space correlated noise has been addressed in \cite{Balan}. 

At this point, let us summarize the strategy that we have followed in order to prove our main result. First of all, we assume that both the drift term and the initial data vanish. Hence, the solution of \eqref{eq: waven} (resp. \eqref{eq: heatn}) is explicitly given by the following mean-zero Gaussian random field:
\[
v^n(t,x)=\int_0^t\int_{\R^d} G_{t-s}(x-y)W^n(ds,dy),\quad (t,x)\in [0,T]\times \R^d,
\]
where $G$ is the fundamental solution of the wave (resp. heat) equation (see Section \ref{sec: main} for the precise formulas). Here, we first show that the family of probability laws of $\{v_n,\, n\geq 1\}$ is tight in the space $\C([0,T]\times \R^d)$. For this, we apply a multidimensional tightness criterion, given in Theorem \ref{thm: criterion}, which seems to be well-known in the literature. Nevertheless, we have not been able to find its proper proof, so we have added it for the sake of completeness (see Appendix \ref{sec: app}). We conclude this part by identifying the limit law. Taking into account that $v^n$ and the limit candidate admit versions with continuous paths, and that both are Gaussian processes, it suffices to show the convergence of the corresponding covariance functions.  

In order to deal with the general case, that is with non-vanishing drift and initial data, we proceed as follows. First, we consider the stochastic wave equation and the stochastic heat equation with bounded drift coefficient. In this cases, we use a path-by-path argument in order to show that the solution $u^n$ has a version with continuous trajectories and that the main result on weak convergence holds. This method is based on showing that $u^n$ can be represented as the image of the stochastic convolution through a certain  continuous functional $F$, almost surely. More precisely,  for any $\eta\in \C([0,T]\times \R^d)$, we define $z:=F(\eta)\in \C([0,T]\times \R^d)$ to be the solution of the following deterministic integral equation:
\beq
z(t,x)=\eta(t,x)+\int_0^t\int_{\R^d} G_{t-s}(x-y)b(z(s,y))dyds,\quad (t,x)\in [0,T]\times \R^d,
\label{eq: 00}
\eeq
where $G$ is the fundamental solution of the wave (resp. heat) equation. This methodology has two important features:

\begin{itemize}
	\item It allows us to prove that $u^n$ and $u$ admit versions with continuous paths, as well as the validity of our main result Theorem \ref{thm: main}, under the minimal assumptions on the initial data. That is, those needed to have existence and uniqueness of solution (see Theorem \ref{thm: existence}). 
	\item We establish two versions of Gronwall lemma adapted to the case of the wave equation and the case of the heat equation with bounded drift, which have interest for itself. These results can be seen as higher dimensional extensions of \cite[Lem. 4.2]{GJQ-Bernoulli} and \cite[Lem. 4.4]{GJQ-Bernoulli}, respectively.	
\end{itemize}  

The above method cannot be applied to the stochastic heat equation with arbitrary Lipschitz drift. The reason is that deterministic integral equation \eqref{eq: 00} is not well-posed in this case.  Instead, our strategy here has been the following. First, we prove that $u^n$ and $u$ admit versions with continuous paths. For this, we need to slightly strengthen the assumptions on the initial condition $u_0$. Next, we show that the family of laws of $\{u^n,\, n\geq 1\}$ is tight in the space $\C([0,T]\times \R^d)$. Finally, we identify the limit law by proving the convergence of the corresponding finite dimensional distributions. For this, we make use of a truncation of the drift $b$ and take advantage of the results for the case of bounded drift. 	

Finally, we also point out that, in the case of the stochastic wave equation \eqref{eq: waven}, we consider space dimensions less than or equal to 3. This is because for higher dimensions the corresponding fundamental solution is a very irregular object, namely a genuine distribution which is not non-negative anymore. Although in the mild form 
of \eqref{eq: waven} it is possible to give a proper sense to the underlying stochastic convolution for any space dimension, it is not clear at all how to deal with the integral term involving the drift $b$ in our setting. It is worth mentioning that this problem was solved in \cite{CD} in the case where the initial data vanish, by making use of the spatially-homogeneous structure of the solution.   

The paper is organized as follows. In Section \ref{sec:hyp}, we introduce the main hypotheses on the spectral measures $\mu_n$, $n\geq 1$, we state the main result of the paper and we provide examples of spectral measures for which our result applies. Section \ref{sec: conv-linear} is devoted to deal with the weak convergence for the linear case. More precisely, the tightness property is studied in Section \ref{sec: tightness}, splitting the computations for wave and heat equations, and the convergence of the corresponding covariance functions is tackled in Section \ref{sec: cov}. In Section \ref{sec: exis-cont}, we deal with the existence and uniqueness of mild solution to equations \eqref{eq: waven} and \eqref{eq: heatn}, and we also prove the corresponding solutions have continuous versions. For the latter to be achieved, we consider the three cases that we already mentioned above: wave equation (Section \ref{sec: wave}), heat equation with bounded drift (Section \ref{sec: heat-bounded}) and heat equation with arbitrary Lipschitz drift (Section \ref{sec: heat-general}).  Finally, Section \ref{sec: proof-main} is devoted to prove the main result of the paper for the general case. In the Appendix, we state and prove a multidimensional tightness criterion which has been applied several times throughout the paper.


\section{Hypotheses, main result and examples}
\label{sec:hyp}

In this section, we first introduce the hypotheses on the family of spectral measures $\{\mu_n,\, n\geq 1\}$ that we will consider, together with two auxiliary results. Next, we define what we understand by the solution to equations \eqref{eq: waven} and \eqref{eq: heatn} and we state the main result of the paper. Finally, we provide examples of spectral measures $\mu_n$ for which our main result applies.    

\subsection{Hypotheses}

This section is devoted to present the hypotheses on the family of spectral measures
$\{\mu_{n},\,n\geq1\}$ that will be considered throughout the paper. We will also provide characterizations of the main hypotheses below which will be useful in some of the main proofs. 

Consider the following assumptions:

\bigskip

\noindent \textbf{(H1)} There exists $q\in(0,2)$ such that 
\beq
\sup_{n\geq1}\int_{\mathbb{R}^{d}}\frac{\mu_{n}(d\xi)}{1+|\xi|^{q}}<\infty.
\label{eq:16}
\eeq

\medskip

\noindent \textbf{(H2)} It holds that 
\[
\lim_{n\to\infty}\int_{\mathbb{R}^{d}}f(\xi)\mu_{n}(d\xi)=\int_{\mathbb{R}^{d}}f(\xi)\mu(d\xi),
\]
for any continuous function $f$ such that 
\begin{equation}
|f(\xi)|\le C\frac{1}{1+|\xi|^{2}},\;\mbox{ for any }\xi\in\mathbb{R}^{d},\label{condAco}
\end{equation}
where $C$ is some positive constant, and $\mu$ is some 
measure on $\mathcal{B}(\mathbb{R}^{d})$.

\medskip

\begin{remark}\label{rmk:5}
	\textbf{(H1)} is equivalent to imposing estimate 
	\eqref{eq:16} with a parameter $q$ as close as we want to $2$. Indeed, assume that 
	hypothesis \textbf{(H1)} holds and take $r\in [q,2)$. Let us verify that 
	\beq
	\sup_{n\geq1}\int_{\mathbb{R}^{d}}\frac{\mu_{n}(d\xi)}{1+|\xi|^r}<\infty.
	\label{eq:20}
	\eeq
	First, we have
	\[
	\sup_{n\geq1}\int_{\{|\xi|\leq 1\}}\frac{\mu_{n}(d\xi)}{1+|\xi|^r}
	\leq \sup_{n\geq1}\int_{\{|\xi|\leq 1\}}\mu_{n}(d\xi) \leq 
	2 \sup_{n\geq1}\int_{\{|\xi|\leq 1\}}\frac{\mu_{n}(d\xi)}{1+|\xi|^q}<\infty.
	\]
	Secondly, since $r\geq q$, it clearly holds that 
	\[
	\sup_{n\geq1}\int_{\{|\xi|>1\}}\frac{\mu_{n}(d\xi)}{1+|\xi|^r}
	\leq  \sup_{n\geq1}\int_{\{|\xi|>1\}}\frac{\mu_{n}(d\xi)}{1+|\xi|^q}<\infty.
	\]
	 \qed
\end{remark}

\medskip

The following lemma verifies that the $\mu$ in {\bf(H2)} is a well-defined spectral measure. 

\begin{lemma}\label{lem: mu}
	Assume that Hypotheses \textbf{(H1)} and \textbf{(H2)} are satisfied. Then, $\mu$ defines a non-negative and symmetric tempered measure and there exists $q\in (0,2)$ such that  
	\beq
	\int_{\R^d} \frac{\mu(d\xi)}{1+|\xi|^q}<\infty.
	\label{eq:472}
	\eeq
	Moreover, $\mu$ is unique. 
\end{lemma}

\begin{proof}
Let us first check the uniqueness. If another measure $\mu'$ satisfies that
\[
\lim_{n\to\infty}\int_{\mathbb{R}^{d}}f(\xi)\mu_{n}(d\xi)=\int_{\mathbb{R}^{d}}f(\xi)\mu'(d\xi),
\]
for any continuous function $f$ satisfying (\ref{condAco}), then
\[
\int_{\mathbb{R}^{d}}f(\xi)d\mu(d\xi)=\int_{\mathbb{R}^{d}}f(\xi)d\mu'(d\xi).
\]
On the other
hand, the indicator function of a rectangle $A=(a_{1},b_{1})\times\cdots\times\left(a_{d},b_{d}\right)$,
${\bf 1}_{A}$, is a pointwise limit of a sequence $\{f_{m},\,m\geq1\}$
of continuous functions with compact support and satisfying $|f_{m}|\le{\bf 1}_{A}$,
for all $m\geq1$. From the above two facts, we can deduce that $\mu=\mu'$
on all the Borel $\sigma$-field. 

It is clear that $\mu$ has to be non-negative. Next, due to the symmetry of $\mu_{n}$, for all $n\geq1$, we first
have that 
\begin{equation}
\int_{\mathbb{R}^{d}}f(\xi)\mu(d\xi)=\int_{\mathbb{R}^{d}}f(-\xi)\mu(d\xi),\label{symm}
\end{equation}
for any continuous function $f$ with compact support. Let $A=(a_{1},b_{1})\times\cdots\times\left(a_{d},b_{d}\right)$
and we will prove that $\mu(A)=\mu(-A)$. Take a sequence $\{f_{m},\,m\geq1\}$
as before. Then, by (\ref{symm}), 
\[
\mu(A)=\lim_{m\to\infty}\int_{\mathbb{R}^{d}}f_{m}(\xi)\mu(d\xi)=\lim_{m\to\infty}\int_{\mathbb{R}^{d}}f_{m}(-\xi)\mu(d\xi).
\]
The latter limit is equal to $\mu(-A)$, because $\lim_{m\to\infty}f_{n}(-\xi)={\bf 1}_{A}(-\xi)={\bf 1}_{-A}(\xi)$.
Hence, $\mu$ is symmetric. 
Finally, we take $q\in (0,2)$ of Hypothesis \textbf{(H1)} and verify that 
\[
\int_{\mathbb{R}^d}\frac{\mu(d\xi)}{1+|\xi|^q}<\infty.
\]
Let $\{f_m,\, m\geq 1\}$ be a sequence of non-negative continuous functions with compact support such that 
\[
\lim_{m\to\infty} f_m(\xi)=\frac{1}{1+|\xi|^q},\quad \text{for all} \; \xi\in \R^d,
\]
and, for any $m\geq 1$,  
\[
f_m(\xi)\leq \frac{1}{1+|\xi|^q},\quad \text{for all} \; \xi\in \R^d.
\]
Then, applying Fatou's lemma and using Hypotheses {\bf (H1)} and {\bf (H2)}, we can argue as follows:
\begin{align*}
	\int_{\mathbb{R}^d}\frac{\mu(d\xi)}{1+|\xi|^q} & \leq 
	\liminf_{m\to\infty} \int_{\mathbb{R}^d} f_m(\xi) \mu(d\xi) \\
	& = \liminf_{m\to\infty} \left( \lim_{n\to\infty} \int_{\mathbb{R}^d} f_m(\xi) \mu_n(d\xi)\right)\\
	& \leq \liminf_{m\to\infty} \left( \sup_{n\geq 1} \int_{\mathbb{R}^d} f_m(\xi) \mu_n(d\xi)\right)\\
	& \leq \sup_{n\geq 1} \int_{\mathbb{R}^d}\frac{\mu_{n}(d\xi)}{1+|\xi|^q} <\infty. 
\end{align*}
Therefore, the proof is complete. 
\end{proof}

\begin{remark}\label{rmk: Dalang}
	Hypothesis {\bf(H1)} and the previous Lemma \ref{lem: mu} imply that $\mu_n$, $n\geq 1$, and $\mu$ satisfy Dalang's condition:
	\[
	\int_{\R^d} \frac{\mu_n(d\xi)}{1+|\xi|^2}<\infty, \quad
	\int_{\R^d} \frac{\mu(d\xi)}{1+|\xi|^2}<\infty.
	\]
\end{remark}	

The following result provides a characterization of hypothesis \textbf{(H1)}
which will be used later on in the paper.

\begin{lemma}\label{lem:1} Hypothesis \textbf{(H1)} is equivalent
to the statement: there exists $q\in (0,2)$ such that the following two conditions are satisfied:  
\begin{enumerate}
\item[(a)] There is a constant $C>0$ such that, for all $h\in(0,1]$, 
\[
\sup_{n\geq1}\mu_{n}(B_{1/h})\le C\,h^{-q},
\]
where $B_{r}:=\{\xi\in\mathbb{R}^{d}, |\xi|\leq r\}$. 
\item[(b)] It holds 
\[
\sup_{n\geq1}\int_{\{|\xi|>1\}}\frac{\mu_{n}(d\xi)}{|\xi|^q}<\infty.
\]
\end{enumerate}
\end{lemma}
\begin{proof}
First, we check that conditions (a) and (b) imply hypothesis \textbf{(H1)}. It holds that 
\[
\int_{\mathbb{R}^{d}}\frac{\mu_{n}(d\xi)}{1+|\xi|^{q}}=\int_{\{|\xi|\le1\}}\frac{\mu_{n}(d\xi)}{1+|\xi|^{q}}+\int_{\{|\xi|>1\}}\frac{\mu_{n}(d\xi)}{1+|\xi|^{q}}=:I_{1}^{n}+I_{2}^{n}.
\]
By (b), we have that 
\[
I_{2}^{n}\le \int_{\{|\xi|>1\}}\frac{\mu_{n}(d\xi)}{|\xi|^{q}}\le C,
\]
for some constant $C$ independent of $n$. On the other hand, 
\[
I_{1}^{n}\le\mu_{n}(B_{1})\le C,
\]
with $C$ independent of $n$, due to condition (a). Therefore, 
\[
\sup_{n\geq1}\int_{\mathbb{R}^{d}}\frac{\mu_{n}(d\xi)}{1+|\xi|^{q}}\le\sup_{n\geq1}I_{1}^{n}+\sup_{n\geq1}I_{2}^{n}<\infty.
\]
We now prove that \textbf{(H1)} implies (a) and (b).
We have, for all $h\in(0,1]$, 
\[
\sup_{n\geq1}\int_{\{|\xi|\le1/h\}}\mu_{n}(d\xi)\le(1+h^{-q})\sup_{n\geq1}\int_{\mathbb{R}^{d}}\frac{\mu_{n}(d\xi)}{1+|\xi|^{q}}\le2\,h^{-q}\sup_{n\geq1}\int_{\mathbb{R}^{d}}\frac{\mu_{n}(d\xi)}{1+|\xi|^{q}}=Ch^{-q},
\]
which implies condition (a). Finally, condition (b) follows from the
estimate 
\[
\sup_{n\geq1}\int_{\{|\xi|>1\}}\frac{\mu_{n}(d\xi)}{|\xi|^{q}}\le 
2\sup_{n\geq1}\int_{\R^d}\frac{\mu_{n}(d\xi)}{1+|\xi|^{q}} <\infty.
\]
\end{proof}

\begin{remark}
We also have a characterization of Hypothesis \textbf{(H2)}. 	
Let $\mu_{n}$ be as before and define the finite measure $\rho_{n}$ as follows: 
\[
\rho_{n}(A):=\int_{A}\frac{\mu_{n}(d\xi)}{1+|\xi|^{2}}, \; \;A\in\mathcal{B}(\mathbb{R}^{d}).
\]
Then, Hypothesis \textbf{(H2)} is equivalent to the fact
that $\rho_{n}$ converges weakly to $\rho$, as $n\to\infty$, where
the measure $\rho$ is given by 
\[
\rho(A):=\int_{A}\frac{\mu(d\xi)}{1+|\xi|^{2}}, \; \; A\in\mathcal{B}(\mathbb{R}^{d}).
\]
That is, 
\[
\lim_{n\to\infty}\int_{\mathbb{R}^{d}}f(\xi)\rho_{n}(d\xi)=\int_{\mathbb{R}^{d}}f(\xi)\rho(d\xi),
\]
for any continuous and bounded function $f$.

\end{remark}



\subsection{Main result}
\label{sec: main}

This section is devoted to state the main result of the paper. Before that, we will define the notion of mild solution to equations \eqref{eq: heatn} and \eqref{eq: waven}, and we will comment on some properties related to the initial data.

We denote by $\{\F^n_t,\, t\ge 0\}$ the filtration generated by $W^n$, which is defined by 
\[
\F_t:=\sigma\left( W^n({\bf 1}_{[0,s]}\varphi),\, s\in [0,t], \varphi\in \mathcal{D}\right)\lor  \mathcal{N},
\] 
where $\mathcal{N}$ denotes the family of $\mathbb{P}$-null sets in $\mathcal{A}$.
The solution to equations \eqref{eq: heatn} and \eqref{eq: waven} will
be interpreted in the {\it{mild}} sense. Namely, for any $T>0$, we
say that an adapted and jointly measurable process
$u^n=\{u^n(t,x),\, (t,x)\in [0,T]\times \R\}$ solves \eqref{eq:
	heatn} (resp. \eqref{eq: waven}) if, for all $(t,x)\in[0,T]\times \R^d$,
it holds
\begin{equation}
	u^n(t,x)  = I^d_0(t,x)+ \int_{0}^{t}\int_{\R^d} G_{t-s}(x-y) W^n(ds,dy)  
	+ \int_{0}^{t} \big(b(u^n(s))*G_{t-s}\big)(x) ds, \quad \mathbb{P}\text{-a.s.},
	\label{eq: mild}
\end{equation}
where $u^n(s)$ denotes the function $u^n(s,\cdot)$.
Moreover, $G$ denotes the fundamental solution of the
heat (resp. wave) equation in $\R^d$, $d\geq 1$ (resp. $d\in \{1,2,3\}$) and $I^d_0(t,x)$ is the solution
of the corresponding deterministic linear equation. In the case of the heat equation, $G$ is the following Gaussian kernel:
\begin{equation}
	G_t(x)=\frac{1}{(2\pi t)^{d/2}}e^{-\frac{|x|^2}{2t}},\; (t,x)\in (0,\infty)\times \R^d.
	\label{eq: heat kernel}
\end{equation}
In the case of the wave equation with $d\in\{1,2\}$, $G$ is the function 
\begin{equation}
	G_t(x)=\begin{cases}
		\frac{1}{2}{\bf 1}_{\{|x|<t\}}(x), & \text{wave equation}\;  d=1, \\
		\frac{1}{2\pi} \frac{1}{\sqrt{t^2-|x|^2}} {\bf 1}_{\{|x|<t\}}(x), & \text{wave equation} \; d=2 \\
	\end{cases}
	\label{eq:1}
\end{equation}
Finally, the fundamental solution of the $3$-dimensional wave equation is given by the measure
\begin{equation}
	G_t(dx)= \frac{1}{4\pi t}\sigma_t(dx), \; t>0,
	\label{eq: 3d}
\end{equation}
where $\sigma_t$ denotes the uniform measure on the $3$-dimensional sphere of radius $t$ (see \cite[Chap. 5]{Folland}). In this case, the second integral in \eqref{eq: mild} is given by
\[
 \int_0^t \big(b(u^n(s))*G_{t-s}\big)(x) ds = 
 \int_0^t \int_{\R^d} b(u^n(s,x-y)) G_{t-s}(dy) ds, 
\]
Still in the case of the wave equation, for any $d\in \{1,2,3\}$, a direct computation based on the expression of $G$ shows that, for all $t>0$,
\begin{equation}\label{eq:56}
 \int_{\R^d} G_t(dx)= t.
\end{equation}

Concerning the term $I^d_0$, it is given by
\begin{equation}
	I^d_0(t,x)=\begin{cases}
		(u_0*G_t)(x), & \text{heat equation}, \\
		(v_0*G_t)(x) + \frac{\partial}{\partial t}(u_0*G_t)(x), & \text{wave equation}.\\
	\end{cases}
	\label{eq:15}
\end{equation}
For the wave equation, it holds (see, for instance, \cite[p. 68-77]{Evans}):
$$ 
I^1_0(t,x)=
\frac{1}{2} \left[ u_0(x+t)+u_0(x-t)\right] +
\frac{1}{2}\int_{x-t}^{x+t} v_0(y)\, dy, \quad (t,x)\in (0,\infty)\times \R,
$$ 
which is the so-called d'Alembert's Formula, 
$$I^2_0(t,x)= \frac{1}{2\pi t} \int_{\{|x-y|<t\}} \frac{u_0(y + t v_0 )+ \nabla u_0(y)\cdot (x-y)}{(t^2- |x-y|^2)^{1/2}}\, dy, \quad (t,x)\in (0,\infty)\times \R^2,$$
and
$$ 
I^3_0(t,x)= \frac{1}{4\pi
	t^2} \int_{\R^3} \left( tv_0(x-y) + u_0(x-y) +  \nabla
u_0(x-y)\cdot y \right)\, \sigma_t(dy), \quad (t,x)\in (0,\infty)\times \R^3. 
$$
In the above formulas, we have implicitly  assumed that all integrals are well-defined. Indeed, 
\cite[Lem. 4.2]{Dalang-Quer} exhibits sufficient conditions on $u_0$ and $v_0$
under which such integrals exist and are uniformly bounded with respect to $t$ and $x$. More precisely, we consider the following hypothesis:

\begin{hypothesis} \label{hyp: ic}
	\hspace{1cm}
	\begin{itemize}
		\item[(i)] Heat equation: $u_0$ is measurable and bounded.
		\item[(ii)] Wave equation: When $d=1$, $u_0$ is bounded and continuous, and $v_0$ is bounded and measurable. When $d=2$, $u_0 \in C^1(\R^2)$ and there is $p \in (2,\infty]$ such that $u_0,\nabla u_0, v_0$ all belong to $L^p(\R^2)$. When $d=3$, $u_0 \in C^1(\R^3)$, $u_0$ and $\nabla u_0$ are bounded, and $v_0$ is bounded and continuous.
	\end{itemize}
\end{hypothesis}	

Then, we have:

\begin{lemma}\label{lem: dqs}(\cite[Lem. 4.2]{Dalang-Quer})	
	Assume that Hypothesis \ref{hyp: ic} holds. Then, $I^d_0$ defines a continuous function such that
	\[
	\sup_{(t,x)\in [0,T]\times \R^d} \big|I_0^d(t,x)\big| <\infty.
	\]
\end{lemma}

In Section \ref{sec: exis-cont}, we will show that equation \eqref{eq: mild} admits a unique solution (see Theorem \ref{thm: existence} for details). At this point, we can state the main result of the paper:

\begin{theorem}\label{thm: main}
	Let $u^n$ be the solution of equation \eqref{eq: mild}, where $G$ is the fundamental solution of the wave equation (resp. heat equation) and $b$ is a Lipschitz function. 
	
	Assume that Hypotheses {\bf{(H1)}} and {\bf{(H2)}} hold, and consider the following assumptions on the initial data:
	\begin{enumerate}
		\item[(a)] Wave equation: (ii) in Hypothesis \ref{hyp: ic}.
		\item[(b)] Heat equation with bounded drift: (i) in Hypothesis \ref{hyp: ic}.
		\item[(c)] Heat equation with general drift: (i) in Hypothesis \ref{hyp: ic} and $u_0\in \C^\alpha(\R^d)$, for some $\alpha\in (0,1)$.
	\end{enumerate}
	Then, as $n\to\infty$, $u^n$  converges in law, in the space $\mathcal{C}([0,T]\times\R^d)$, to the random field $u$ which solves the equation
	\begin{equation}
		u(t,x)  = I^d_0(t,x)+ \int_{0}^{t}\int_{\R^d} G_{t-s}(x-y) W(ds,dy)  
		+ \int_{0}^{t} \big(b(u)*G_{t-s}\big)(x) ds, 
		\label{eq: mild_u}
	\end{equation} 
for all $(t,x)\in [0,T]\times \R^d$,
	where $W$ denotes a Gaussian spatially homogeneous noise with spectral measure $\mu$ (defined in Hypothesis {\bf{(H2)}}). 
\end{theorem}

\begin{remark}\label{rmk:10}
	The proof of Theorem \ref{thm: existence} works for equation \eqref{eq: mild_u} as well. That is, for any $p\geq 1$, the latter equation admits a unique solution in the space of $L^2(\Omega)$-continuous and adapted processes such that
	\beq
	\sup_{(t,x)\in [0,T]\times \R^d}
	\E{|u(t,x)|^p}<\infty.	
	\label{eq:789}
	\eeq

\end{remark}


\subsection{Examples}
\label{sec: examples}

This section is devoted to present some examples of families of spectral measures $\{\mu_n,\, n\geq 1\}$ for which Theorem \ref{thm: main} holds. 

\subsubsection{Fractional noise with $d=1$}
\label{sec: frac1}

We prove that Hypotheses \textbf{(H1)} and \textbf{(H2)} are satisfied
in the case where our noise is fractional in space and $d=1$. More precisely,
assume that 
\[
\mu_{n}(d\xi)=C_{H_{n}}|\xi|^{1-2H_{n}}d\xi, \quad \xi\in \R,
\]
with $H_{n}\in(0,1)$ and 
\begin{equation}
C_{H}=\frac{\Gamma(2H+1)\sin(\pi H)}{2\pi},\;H\in(0,1).
\label{eq:14}
\end{equation}
We suppose that $H_{n}\to H_{0}\in(0,1)$, as $n$ tends to infinity.
Then, the measure $\mu$ will be given by 
\[
\mu(d\xi)=C_{H_{0}}|\xi|^{1-2H_{0}}\,d\xi.
\]
First, we will check \textbf{(H1)}. Define 
\[
\hmin:=\inf_{n\geq1}H_{n}\in(0,1)\quad\mbox{and}\quad\hmax:=\sup_{n\geq1}H_{n}\in(0,1).
\]
Take $q\in(2-2\hmin,2)$. It holds 
\[
\int_{\mathbb{R}}\frac{|\xi|^{1-2H_{n}}}{1+|\xi|^{q}}\,d\xi=\int_{\{|\xi|\le1\}}\frac{|\xi|^{1-2H_{n}}}{1+|\xi|^{q}}\,d\xi+\int_{\{|\xi|>1\}}\frac{|\xi|^{1-2H_{n}}}{1+|\xi|^{q}}\,d\xi=:I_{1}^{n}+I_{2}^{n}.
\]
We have that 
\[
I_{1}^{n}\le2\int_{0}^{1}\xi^{1-2H_{n}}\,d\xi=\frac{2}{2-2H_{n}}\le\frac{1}{1-\hmax}.
\]
On the other hand, 
\[
I_{2}^{n}\le\int_{\{|\xi|>1\}}\frac{\axi^{1-2\hmin}}{1+\axi^{q}}d\xi\le C\int_{\{|\xi|>1\}}\frac{1}{\axi^{q+2\hmin-1}}d\xi\le C,
\]
because $q+2\hmin-1>1$. The above  two inequalities and the fact that
the constants $C_{H_{n}}$ are bounded (since the function $\Gamma$
is continuous in $[1,3]$) prove that \textbf{(H1)} is satisfied.

Now, we will prove that \textbf{(H2)} is fulfilled. Suppose that $f$
is a continuous function satisfying (\ref{condAco}). Since the constant
$C_{H}$ is a continuous function of $H$, we must prove that 
\[
\lim_{n\to\infty}\int_{\R}f(\xi)\axi^{1-2H_{n}}\,d\xi=\int_{\R}f(\xi)\axi^{1-2H_{0}}\,d\xi.
\]
On the one hand, by the dominated convergence theorem, we have 
\[
\lim_{n\to\infty}\int_{\{|\xi|\le1\}}f(\xi)\axi^{1-2H_{n}}\,d\xi=\int_{\{|\xi|\le1\}}f(\xi)\axi^{1-2H_{0}}\,d\xi,
\]
because, when $\axi\le1$, 
\[
|f(\xi)|\,\axi^{1-2H_{n}}\le C\axi^{1-2\hmax}.
\]
On the other hand, again applying the dominated convergence theorem,
it holds 
\[
\lim_{n\to\infty}\int_{\{|\xi|>1\}}f(\xi)\axi^{1-2H_{n}}\,d\xi=\int_{\{|\xi|>1\}}f(\xi)\axi^{1-2H_{0}}\,d\xi,
\]
since in the case $\axi>1$ we have 
\[
|f(\xi)|\,\axi^{1-2H_{n}}\le C\axi^{-1-2\hmin},
\]
by using condition (\ref{condAco}). This concludes the proof. \qed

\subsubsection{The anisotropic fractional noise}
\label{sec: ani}

We consider a  Gaussian spatially homogeneous noise which is white in time and anisotropic fractional in space. This noise depends on a $d$-dimensional parameter $H=(H_1,\ldots,H_d)\in (0,1)^d$, and the corresponding spectral measure is given by
$$\mu_H(d\xi)=\prod_{j=1}^dC_{H_j} |\xi_j|^{1-2H_j}d\xi, \quad \xi\in \R^d.$$
Here,
$$C_{H_j}=\frac{\Gamma(2H_j+1)\sin(\pi H_j)}{2\pi}.$$
We note that $\mu_H$ is the spectral measure associated to the covariance of the anisotropic fractional Brownian sheet. 
We will see that, under certain hypotheses, if we have a sequence of parameters $\{H_n\}_{n\geq 1}$ satisfying $H_n\to H_0$, then the family  of measures $\{\mu_n\}_{n\geq 1}$ defined by $\mu_n:=\mu_{H_n}$   satisfies hypotheses {\bf(H1)} and {\bf (H2)} with $\mu=\mu_{H_0}$. The needed conditions is essentially the same as that imposed to ensure that $\mu_n$ satisfies Dalang's condition.
We assume that $d\ge 2$, since the case $d=1$ has already been treated in Section \ref{sec: frac1}.

Let $H=(H_1,\ldots, H_d)\in (0,1)^d$. We first check under which hypotheses Dalang's condition is satisfied for $\mu_H$, that is:
\begin{equation}\label{DalangFracSheet}
	\int_{\mathbb R^d}\frac{\prod_{j=1}^d |\xi_j|^{1-2H_j}}{1+|\xi|^2}\,d\xi<\infty.
\end{equation}
We consider the following $d$-dimensional spherical coordinates:
\begin{align*}
	&\xi_1=r\,\sin\theta_1\,\sin\theta_2\cdots \sin\theta_{d-1}\\
	&\xi_2=r\,\sin\theta_1\,\sin\theta_2\cdots \sin\theta_{d-2}\cos\theta_{d-1}\\
	&\xi_3=r\,\sin\theta_1\,\sin\theta_2\cdots \sin\theta_{d-3}\cos\theta_{d-2}\\
	&\phantom{xxxxxxx}\vdots\\
	&\xi_d=r\cos\theta_1,		
\end{align*}
with $\theta_j\in (0,\pi)$, for $j=1,\ldots,d-2$ and $\theta_{d-1}\in (0,2\pi)$.  The Jacobian of the underlying change of variables is given by
$$J(r,\theta_1,\dots,\theta_{d-1})=r^{d-1}\sin^{d-2}\theta_1\,\sin^{d-3}\theta_2\cdots\sin\theta_{d-2}.   $$
Performing the change of variables, the integral of (\ref{DalangFracSheet}) becomes
$$\int_0^\infty \int_{(0,\pi)^{d-2}\times(0,2\pi)}\dfrac{r^{d-2\sum_{j=1}^d  H_j}\,r^{d-1}}{1+r^2}\,f(\theta_1\ldots\theta_{d-1})\,d\theta_1\cdots d\theta_{d-1}dr,
$$
where
\begin{align*}
	f(\theta_1,\ldots,\theta_{d-1})=&|\sin\theta_1|^{d-2+\sum_{j=1}^{d-1}(1-2H_j)}\,|\sin\theta_2|^{d-3+\sum_{j=1}^{d-2}(1-2H_j)}
	\times\cdots \times|\sin\theta_{d-2}|^{1+\sum_{j=1}^{2}(1-2H_j)}\\
	& \qquad \times|\cos\theta_{d-1}|^{1-2H_2}\,|\cos\theta_{d-2}|^{1-2H_3}\cdots |\cos\theta_{1}|^{1-2H_d}.
\end{align*}
This function is integrable because all the exponents of the trigonometrical functions are  greater than $-1$.
On the other hand, in order that the integral with respect to $r$ is finite, we need that
$$d-2\sum_{j=1}^d  H_j+d-1>-1,$$
which is satisfied, and that
$$d-2\sum_{j=1}^d H_j+d-1-2<-1.$$
The latter condition is satisfied if and only if
\begin{equation*}
	\sum_{j=1}^d  H_j>d-1.
\end{equation*}
At this point, we go back to the sequence of spectral measures given by
\beq
\mu_n(d\xi)=\prod_{j=1}^d C_{H_j^n} |\xi_j|^{1-2H_j^n}d\xi, \quad \xi\in \R^d.
\label{eq:347}
\eeq
We assume that the sequence of parameters $\{H_n=(H_1^n,\ldots,H_d^n)\}_{n\geq 1}$ satisfies the following:
\begin{enumerate}[(i)]
	\item For all $n\ge 1$,
	$$	\sum_{j=1}^d H^n_j>d-1.$$
	\item It holds
	$$\lim_{n\to\infty} H_n = H_0=(H_1^0,\ldots,H_d^0) \quad \mbox{and} \quad \sum_{j=1}^d H_j^0>d-1.$$
\end{enumerate}	
We show that, under conditions (i) and (ii) above, hypotheses {\bf (H1)} and {\bf (H2)} are satisfied. 

We start with hypothesis {\bf (H1)}.
Set $A=\{(x_1,\ldots,x_d)\in  (0,1)^d,\sum_{j=1}^d x_j >d-1\}$. Since $H_n\in A$ for all $n\ge 1$,  $H_0\in A$ and $H_n\to H_0$, we have that
$$L:=\inf_{n\geq 1}\sum_{j=1}^dH_j^n>d-1,$$
and 
$$U:=\sup_{n\geq 1}\sup_{j=1,\dots,d}H_j^n<1.$$
Observe that $0<2d-2L<2$, so we will prove that hypothesis {\bf (H1)} is satisfied taking $q\in(2d-2L,\,2)$. That is, we will check that
$$\sup_{n\geq 1}\int_{\mathbb R^d}\frac{\mu_n(d\xi)}{1+|\xi|^q}<\infty.$$
First, note that the product of constants $\prod_{j=1}^dC_{H_j^n}$  is bounded because the function $\Gamma$ is continuous on the interval $[1,3]$. Thus, we must study the term:
$$\sup_{n\geq 1}\int_{\mathbb R^d}\prod_{j=1}^d|\xi_j|^{1-2H_j^n}\frac{d\xi}{1+|\xi|^q}.$$
Performing the change of variables to spherical coordinates, the last quantity equals to
\begin{align}
	& \sup_{n\geq 1}\int_0^\infty \int_{(0,\pi)^{d-2}\times (0,2\pi)} \left(\prod_{j=1}^{d-1}|\cos\theta_{d-j}|^{1-2H_j^n}\right)
	\left(
	\prod_{j=1}^{d-2}|\sin\theta_{d-j-1}|^{j+\sum_{k=1}^{j+1}(1-2H_k^n)}\right) \nonumber \\ 
	& \qquad \qquad \times r^{d-1+\sum_{j=1}^d (1-2H_j^n)}\,\frac1{1+r^q}\,d\theta_1\cdots d\theta_{d-1}dr.
	\label{eq:374}
\end{align}
We can bound the trigonometrical part of the above integral in the following way:
\begin{align*}
	& \left(\prod_{j=1}^{d-1}|\cos\theta_{d-j}|^{1-2H_j^n}\right)
	\left( \prod_{j=1}^{d-2}|\sin\theta_{d-j-1}|^{j+\sum_{k=1}^{j+1}(1-2H_k^n)}\right) \\ 
	& \qquad \qquad \le \left(\prod_{j=1}^{d-1}|\cos\theta_{d-j}|^{1-2U}\right) \left(\prod_{j=1}^{d-2}|\sin\theta_{d-j-1}|^{j+(j+1)(1-2U)}\right).
\end{align*}
Due to the fact that $1-2U>-1$ and $j+(j+1)(1-2U)>-1$, for all $j=1,\dots,d-2$, the integral of this part in \eqref{eq:374} is bounded, independently of $n$.
Now, we consider the integral in \eqref{eq:374} corresponding to the radial part:
\begin{align*}
	\int_0^{\infty} r^{2d-2\sum_{j=1}^d H_j^n-1}\,\frac1{1+r^q}dr 
& \le\int_0^1 r^{2d-2\sum_{j=1}^d H_j^n-1}dr+\int_1^\infty r^{2d-2\sum_{j=1}^d H_j^n-1-q}dr\\
& =:I_1^n\,+\,I_2^n.
\end{align*} 
We have that
$$\sup_{n\geq 1} I_1^n\le\int_0^1 r^{2d-2dU-1}\,dr<\infty,$$
because $2d-2U-1>-1$,
and 
$$\sup_{n\geq 1} I_2^n\le\int_0^\infty r^{2d-2L-1-q}\, dr<\infty,$$
since $2d-2L-1-q<-1$. This concludes that $\{\mu_n\}_{n\geq 1}$ given by \eqref{eq:347} satisfies {\bf (H1)}.

Next, we check that hypothesis {\bf (H2)} is fulfilled.
Let
$$\mu(d\xi)=
\prod_{j=1}^d C_{H_j^0}  |\xi_j|^{1-2H_j^0} d\xi.$$
We must see that, for any continuous function $f:\R^d\to\R$ satisfying
$$|f(\xi)|\le \frac{C}{1+|\xi|^2},$$
we have 
\begin{equation}\label{converg}
	\lim_{n\to\infty} \prod_{j=1}^d C_{H_j^n}\int_{\mathbb R^d}f(\xi)|\xi_1|^{1-2H_1^n}\cdots|\xi_d|^{1-2H_d^n} d\xi = 
	\prod_{j=1}^d C_{H_j^0}\int_{\mathbb R^d}f(\xi)|\xi_1|^{1-2H_1^0}\cdots|\xi_d|^{1-2H_d^0}d\xi.
\end{equation}
Due to the continuity of $C_{H_j}$ with respect to the parameter $H_j$, we have the convergence of the above product of constants.
On the other hand, by the dominated convergence theorem,
$$
\lim_{n\to\infty} \int_{\{|\xi|\le 1\}}f(\xi)|\xi_1|^{1-2H_1^n}\cdots |\xi_d|^{1-2H_d^n} d\xi = \int_{\{|\xi|\le 1\}}f(\xi)|\xi_1|^{1-2H_1^0}\cdots |\xi_d|^{1-2H_d^0} d\xi.$$ 
Indeed, in the domain $\{|\xi|\le 1\}$ it holds that $|\xi_j|\le 1$, for any $j=1,\ldots,d$, and therefore
$$|f(\xi)|\,|\xi_1|^{1-2H_1^n}\cdots |\xi_d|^{1-2H_d^n}\le |f(\xi)|\prod_{j=1}^d|\xi_j|^{1-2U},$$
which is an integrable function on $[-1,1]^d$ and, thus, on $\{\xi\in \R^d, |\xi|\le 1\}$ as well. 
Finally, by passing to spherical coordinates, we can write
\begin{align*}
	& \int_{\{|\xi|> 1\}}f(\xi) |\xi_1|^{1-2H_1^n}\cdots |\xi_d|^{1-2H_d^n}d\xi \\
	& \qquad =
	\int_1^\infty \int_{(0,\pi)^{d-2}\times (0,2\pi)} \left( \prod_{j=1}^{d-1}|\cos\theta_{d-j}|^{1-2H_j^n}\right)\left(
	\prod_{j=1}^{d-2}|\sin\theta_{d-j-1}|^{j+\sum_{k=1}^{j+1}(1-2H_k^n)}\right) \\
	& \qquad \qquad 
	\times g(\theta_1,\ldots,\theta_{d-1},r) \,r^{d-1+\sum_{j=1}^d (1-2H_j^n)} d\theta_1\cdots d\theta_{d-1}dr,
\end{align*}
where $g$ is the function $f$ expressed in terms of the spherical coordinates.
We can also apply the dominated convergence theorem and obtain that the last integral converges, as $n\to\infty$, to the same expression but replacing $H_j^n$ by $H_0^n$. In fact, it holds that
\begin{align*}
	& \left( \prod_{j=1}^{d-1}|\cos\theta_{d-j}|^{1-2H_j^n}\right)
	\left(
	\prod_{j=1}^{d-2}|\sin\theta_{d-j-1}|^{j+\sum_{k=1}^{j+1}(1-2H_k^n)}\right) 
	 |g(\theta_1,\ldots,\theta_{d-1},r)|\,r^{d-1+\sum_{j=1}^d (1-2H_j^n)}\\
	 & \qquad \le C \left(\prod_{j=1}^{d-1}|\cos\theta_{d-j}|^{1-2U}\right)
	 \left(
	\prod_{j=1}^{d-2}|\sin\theta_{d-j-1}|^{j+(j+1)(1-2U)}\right)
	\frac{r^{2d-2L-1}}{1+r^2}\\ 
	& \qquad \le C \left( \prod_{j=1}^{d-1}|\cos\theta_{d-j}|^{1-2U}\right) 
	\left(
	\prod_{j=1}^{d-2}|\sin\theta_{d-j-1}|^{j+(j+1)(1-2U)}\right) r^{2d-2L-1-q}.
\end{align*}
As we have seen before, the latter expression defines an integrable function. This concludes that {\bf (H2)} is satisfied.


\subsubsection{The isotropic fractional noise}
\label{sec: iso}

We now consider a Gaussian spatially homogeneous noise which is white in time and isotropic fractional in space. That is, it is the noise associated to a centered Gaussian random field $\{X^H(t,x),\, (t,x)\in \R_+\times \R^d\}$ with covariance function given by
$$\E{X^H(s,x)X^H(t,y)}=\frac{\sigma_0^2}2\min(s,t)(|x|^{2H}+|y|^{2H}-|x-y|^{2H}),$$
where $H\in (0,1)$ and $\sigma_0^2$ is some positive constant.
A centered Gaussian random field $\{Y^H(x),\, x\in\R^d\}$ with covariance function given by
$$\E{Y^H(x)Y^H(y)}=\frac{\sigma^2_0}2(|x|^{2H}+|y|^{2H}-|x-y|^{2H})$$
is called   {\it isotropic fractional Brownian sheet} or also {\it Lévy fractional Brownian sheet}. It is the only (modulo multiplicative constants) $H$-self-similar random field with stationary increments in the strong sense, that is
$$\{Y^H(g(x))-Y^H(g(0)),\,x\in\mathbb R^d\}\overset{\mathcal{L}}{=}\{Y^H(x)-Y^H(0),\,x\in\mathbb R^d\},$$
for any Euclidian rigid body motions $g$, which form a group and are defined as compositions of rotations and translations (see, for instance, \cite[Sec. 7.2 and 8.1]{Samorod}).
It can be proved that, up to a multiplicative constant, the Lévy fractional Brownian sheet has the following spectral representation in law:
\beq
Y^H_x \stackrel{\mathcal{L}}{=}  \int_{\mathbb R^d}\frac{e^{i <x,\xi>}-1}{|\xi|^{H+\frac{d}2}}\,\hat{W}(d\xi), \quad x\in \R^d,
\label{eq:564}
\eeq
where $\hat{W}$ is a complex Brownian measure on $\mathbb R^d$. In fact, it is easily seen that the right hand-side above is self-similar of index $H$ and has stationary increments in the strong sense.
From the spectral representation \eqref{eq:564}, we can compute the underlying spectral measure. First, we have that, for any rectangle $(x,x']\subset \R^d$, with $x,x'\in \R^d$,
$$\mathcal F({\bf 1}_{(x,x']})(y) =(-i)^d\prod_{k=1}^d (y_k^{-1})\,\Delta_{(x,x']}e^{i<\cdot,y>},$$
where $\Delta_{(x,x']}f(\cdot)$ denotes the rectangular increment of the function $f:\R^d\to\R$ on $(x,x']$.
Hence, by \eqref{eq:564},
$$\Delta_{(x,x']}Y^H=\int_{\R^d}\mathcal F({\bf 1}_{(x,x']})(\xi)\,i^d \frac{\prod_{k=1}^d \xi_k}{|\xi|^{H+\frac{d}2}}\hat{W}(d\xi).$$
Making an abuse of notation, we set $Y^H({\bf 1}_{(x,x']}):=\Delta_{(x,x']}Y^H$ and extend this definition by linearity  to any elementary function $\phi$ (finite linear combinations of indicator functions of rectangles):
$$Y^H(\phi)=\int_{\R^d}\mathcal F(\phi)(\xi)\,i^d  \frac{\prod_{k=1}^d \xi_k}{|\xi|^{H+\frac{d}2}}\,\hat{W}(d\xi).$$
Computing the covariance functional of the map $Y^H$, we obtain that its associated spectral measure is given by
$$\mu^H(d\xi)=\frac{\prod_{k=1}^d\xi_k^2}{|\xi|^{2H+d}} d\xi,\quad \xi\in \R^d.$$
By using a change of variables with spherical coordinates, it can be checked that the measure $\mu^H$  does not satisfy Dalang's condition unless $d=1$, which corresponds to the fractional  noise studied in Section \ref{sec: frac1}.

\subsubsection{Riesz kernel}
\label{sec: riesz}

For any $\alpha\in (0,d)$ set $f_\alpha(x)=|x|^{-\alpha}$, which is called the Riesz kernel of order $\alpha$. We have that this function defines a covariance functional given by
$$\int_0^\infty\int_{\R^d} \int_{\R^d} \varphi(t,x)f_\alpha(x-y)\psi(t,y)dxdydt,$$ for any $\varphi,\psi\in\mathcal D=\mathcal{C}_{0}^{\infty}([0,\infty)\times\mathbb{R}^{d})$.
It is well-known that the above functional can be expressed in the form
$$\int_0^\infty\int_{\R^d}\mathcal F\varphi(t,\cdot)(\xi) \overline{\mathcal F\psi(t,\cdot)(\xi)}\mu_\alpha(d\xi)dt,
$$
where
$$\mu_\alpha(d\xi)=c_\alpha f_{d-\alpha}(\xi)d\xi=c_\alpha |\xi|^{\alpha-d}d\xi
$$
and the constant $c_\alpha$ is given by 
$$c_\alpha=\frac{\Gamma(\frac{d-\alpha}2)}{ 2^\alpha\pi^{d/2}\Gamma(\frac{d}2)}.$$

When $d=1$ the Riesz kernel is, modulo a multiplicative constant, a particular case of the fractional noise presented in Section \ref{sec: frac1}. More precisely, it corresponds to a fractional noise with $H=1-\frac{\alpha}{2}\in (\frac12,1)$. Note that the fractional noise can be also considered for $H\in (0,\frac12]$, and in this case the Riesz kernel would not be given by a function but a genuine distribution (see Section \ref{sec: frac1}).

We will now deal the with the case $d\ge 2$. It is readily checked that, to ensure that $\mu_\alpha$ satisfies Dalang's condition, we must have that $\alpha<2$.
Consider a sequence $\{\alpha_n\}_{n\geq 1}$ such that $\alpha_n\in (0, 2)$, for all $n\geq 1$, and satisfying $\alpha_n\rightarrow \alpha_0$, as $n\to\infty$, for some $\alpha_0\in (0,2)$. 
Then, taking $\mu_n:=\mu_{\alpha_n}$ and $\mu:=\mu_{\alpha_0}$, hypotheses {\bf (H1)} and {\bf (H2)} are satisfied taking $q\in(\sup_{n\geq 1} \alpha_n, 2)$. The proof follows easily by using that the constant $c_\alpha$ defines a continuous function of $\alpha$ and  that
$$\inf_{n\geq 1} \alpha_n>0\qquad\text{and}\qquad \sup_{n\geq 1}\alpha_n<2.$$


\section{Weak convergence for the linear case}
\label{sec: conv-linear}

In this section, we consider equations \eqref{eq: heatn} and \eqref{eq: waven}
in the case where the drift term $b$ and the initial data vanish.
This implies that the solution of these equations is explicitly given
by 
\beq
v^{n}(t,x):=\int_{0}^{t}\int_{\R^d}G_{t-s}(x-y)W^{n}(ds,dy),\quad(t,x)\in(0,T]\times\R^d,
\label{eq: stochconv}
\eeq
where we recall that $G$ is the fundamental solution of the heat (respectively wave)
equation on $\R^d$ (see \eqref{eq: heat kernel}-\eqref{eq: 3d}). Note that $v^{n}$ defines a mean-zero Gaussian
process such that 
\[
\E{|v^{n}(t,x)|^{2}}=\int_{0}^{t}\int_{\R^d}|\F G_{t-s}(x-\cdot)(\xi)|^{2}\mu_{n}(d\xi)ds=\int_{0}^{t}\int_{\R^d}|\F G_s(\xi)|^{2}\mu_{n}(d\xi)ds,
\]
where we have used that $\F G_{t-s}(x-\cdot)(\xi)=\F G_{t-s}(\cdot-x)(-\xi)=e^{-i<x,\xi>}\overline {\F G_{t-s}(\xi)}$. Moreover, we have the following uniform estimate for the moments of $v^n$:

\begin{lemma}\label{lem: unif-moment}
Assume that Hypothesis {\bf (H1)} is satisfied. Then, for all $p\geq 1$, 
\[
\sup_{n\geq 1}\sup_{(t,x)\in [0,T]\times \R^d} \E{|v^{n}(t,x)|^p}<\infty.
\]
\end{lemma}  

\begin{proof}
	Let $(t,x)\in [0,T]\times \R^d$. Owing to Examples 6 and 8 in \cite{DalangEJP}, and taking into account that the parameter $q$ of Hypothesis {\bf (H1)} satisfies $q\in (0,2)$, we have
\begin{align*}
	\E{\left|\int_0^t\int_{\R^d} G_{t-s}(x-y)W^n(ds,dy)\right|^p} & 
	= C  \left(\int_0^t\int_{\R^d} |\mathcal{F}G_{t-s}(\xi)|^2 \mu_n(d\xi) ds\right)^{\frac p2}\\
	& = C \left(\int_0^t\int_{\R^d} |\mathcal{F}G_s(\xi)|^2 \mu_n(d\xi) ds\right)^{\frac p2}\\
	& \leq C \left(\int_{\R^d} \frac{\mu_n(d\xi)}{1+|\xi|^2} \right)^{\frac p2} \\
	& \leq C \left(\int_{\R^d} \frac{\mu_n(d\xi)}{1+|\xi|^q} \right)^{\frac p2} \\
	& \leq C \left(\sup_{n\geq 1} \int_{\R^d} \frac{\mu_n(d\xi)}{1+|\xi|^q} \right)^{\frac p2}.
\end{align*}
The above supremum is finite, by Hypothesis {\bf (H1)}, which concludes the proof. 
\end{proof}

This section is devoted to prove the following result, which corresponds to Theorem \ref{thm: main} for the linear case. 

\begin{theorem}\label{thm: linear}
	Let $v^n$ be the random field defined by \eqref{eq: stochconv}, where $G$ is the fundamental solution of the wave equation (respect. heat equation). Assume that Hypotheses {\bf (H1)} and {\bf (H2)} hold. Then, as $n\to\infty$, $v^n$ converges in law, in the space $\C([0,T]\times \R^d)$, to the random field
	\beq
	v(t,x)=\int_0^t\int_{\R^d} G_{t-s}(x-y)W(ds,dy), \quad (t,x)\in [0,T]\times \R^d,
	\label{eq:822}
	\eeq
	where $W$ is a Gaussian spatially homogeneous noise with spectral measure $\mu$ (defined in Hypothesis {\bf (H2)}).
\end{theorem}

\begin{proof}
First, we check that the family of laws of $\{v^n,\, n\geq 1\}$ is tight in the space $\mathcal{C}([0,T]\times \R^d)$. This is shown in Proposition \ref{prop:tightness}, from which we also deduce that $v^n$ has a version with continuous paths, for all $n\geq 1$. Secondly, 
as a consequence of Proposition \ref{prop: vhol}, we have that $v$ is a well-defined random variable taking values in $\C([0,T]\times \R^d)$. Finally,   
we identify the limit law by proving that the finite-dimensional distributions of $v^n$ converge to those of $v$, as $n\to\infty$. This is an immediate consequence of Proposition \ref{prop:cov},
where we show that the covariance function of $v^n$ converges to that of $v$, taking into account that both $v^n$ and $v$ are centered Gaussian random fields. 
\end{proof}
 

\subsection{Tightness}
\label{sec: tightness}

In this section, we aim to prove the following result:

\begin{proposition}\label{prop:tightness} 
	Let $v^n$ be the random field defined by \eqref{eq: stochconv}, where $G$ is the fundamental solution of the wave equation (respect. heat equation).
Assume that hypothesis \textbf{(H1)} holds true. Then, the following are satisfied:
\begin{itemize}
	\item[(a)] For any compact $K\subset \R^d$, there is a constant $C>0$ such that, for all $x,z\in K$,
\beq
\sup_{n\geq1} \sup_{t\in[0,T]} \E{|v^{n}(t,x)-v^{n}(t,z)|^{2}}\leq C|x-z|^{2-q}.
\label{eq:738}
\eeq

\item[(b)] There exists a constant $C>0$ such that, for any $s,t \in [0,T]$,
\beq
\sup_{n\geq1} \sup_{x\in \R^d} \E{|v^{n}(t,x)-v^{n}(s,x)|^{2}}
\leq
\left\{
\begin{array}{cc}
	|t-s|^{2-q}, \quad \text{wave equation},\\ [0.1cm]
	|t-s|^{1-\frac{q}{2}}, \quad \text{heat equation}.
\end{array}
\right.
\label{eq:739}
\eeq
\end{itemize}
Moreover, the laws of $\{v^n,\,n\geq1\}$
form a tight family in the space $\mathcal{C}([0,T]\times\R^d)$. 
\end{proposition}

In the following two subsections, we will prove the above proposition separately for the wave equation (Section \ref{sec: tightness-wave}) and the heat equation (\ref{sec: tightness-heat}). 
Moreover, the proof of Proposition \ref{prop:tightness} can be easily adapted to show that the random field $v$ given in \eqref{eq:822} satisfies estimates \eqref{eq:738} and \eqref{eq:739};
recall that, owing to Lemma \ref{lem: mu}, the measure $\mu$ satisfies condition \eqref{eq:472}.  Hence, Kolmogorov's continuity criterion implies that $v$ has a modification with (Hölder-)continuous paths. These statements can be summarized in the following result:  

\begin{proposition}\label{prop: vhol}
	Let $v$ be the random field defined by \eqref{eq:822}, where $G$ is the fundamental solution of the wave equation (resp. heat equation). Assume that Hypotheses \textbf{(H1)} and \textbf{(H2)} are satisfied. Then, it holds:
	\begin{itemize}
		\item[(a)] For any compact $K\subset \R^d$, there is a constant $C>0$ such that, for all $x,z\in K$,
		$$
		\sup_{t\in[0,T]} \E{|v(t,x)-v(t,z)|^{2}}\leq C|x-z|^{2-q}.
		$$
		
		\item[(b)] There exists a constant $C>0$ such that, for any $s,t \in [0,T]$,
		$$
		\sup_{x\in \R^d} \E{|v(t,x)-v(s,x)|^{2}}
		\leq
		\left\{
		\begin{array}{cc}
			|t-s|^{2-q}, \quad \text{wave equation},\\ [0.1cm]
			|t-s|^{1-\frac{q}{2}}, \quad \text{heat equation}.
		\end{array}
		\right.
		$$	
	\end{itemize}
	Furthermore, $v$ has a version with (Hölder-)continuous paths. 
\end{proposition}


\subsubsection{Wave equation}

\label{sec: tightness-wave}

Here, we prove Proposition \ref{prop:tightness} in the case where $G$ in \eqref{eq: stochconv} is the fundamental solution of the wave equation in $\R^d$, $d\in \{1,2,3\}$. In this case, we recall that, for all $t>0$, the Fourier transform of $G_t$ admits indeed a unified expression for all dimensions, which is the following: 
\[
\F G_{t}(\xi)=\frac{\sin(t|\xi|)}{|\xi|},\quad t>0,\;\xi\in\R^d.
\]
Let us first analyze the square moment of the space increments of
$v^{n}$. Let $t\in(0,T]$ (the case $t=0$ is trivial) and $x,z\in\R^d$,
define $h:=z-x$ and assume that $|h|\in(0,1)$. Then, 
\begin{align}
\E{|v^{n}(t,x)-v^{n}(t,z)|^{2}} & =\int_{0}^{t}\int_{\R^d}\big|\F\big(G_{t-s}(x-\cdot)-G_{t-s}(z-\cdot)\big)(\xi)\big|^{2}\mu_{n}(d\xi)ds\nonumber \\
 & =\int_{0}^{t}\int_{\R^d}\big|1-e^{-i<\xi, h>}\big|^{2}|\F G_{t-s}(\xi)|^{2}\mu_{n}(d\xi)ds \nonumber \\
 & =2\int_{0}^{t}\int_{\R^d}\big(1-\cos(<\xi, h>)\big)\frac{\sin^{2}((t-s)|\xi|)}{|\xi|^{2}}\mu_{n}(d\xi)ds \nonumber \\
 & \leq2\int_{0}^{T}\int_{\R^d}\big(1-\cos(<\xi, h>)\big)\frac{\sin^{2}(s|\xi|)}{|\xi|^{2}}\mu_{n}(d\xi)ds.
 \label{eq:99}
\end{align}
Applying the inequality $1-\cos(x)\leq\frac{x^{2}}{2}$, which holds
for any $x\in\R$, and (a) of Lemma \ref{lem:1}, we have 
\begin{align}
2\int_{0}^{T}\int_{\{|\xi|\leq\frac{1}{|h|}\}}\big(1-\cos(<\xi, h>)\big)\frac{\sin^{2}(s|\xi|)}{|\xi|^{2}}\mu_{n}(d\xi)ds & \leq |h|^2\int_{0}^{T}\int_{\{|\xi|\leq\frac{1}{|h|}\}}\sin^{2}(s|\xi|)\mu_{n}(d\xi)ds\nonumber \\
 & \leq T |h|^{2}\sup_{n\geq1}\mu_{n}(B_{1/|h|})\nonumber \\
 & \leq C|h|^{2-q}.
 \label{eq:111}
\end{align}
On the other hand, owing to (b) in Lemma \ref{lem:1},
it holds 
\begin{align}
2\int_{0}^{T}\int_{\{|\xi|>\frac{1}{|h|}\}}\big(1-\cos(<\xi, h>)\big)\frac{\sin^{2}(s|\xi|)}{|\xi|^{2}}\mu_{n}(d\xi)ds & \leq4T\int_{\{|\xi|>\frac{1}{|h|}\}}\frac{\mu_{n}(d\xi)}{|\xi|^{2}}\nonumber \\
 & \leq4T |h|^{2-q}\sup_{n\geq1}\int_{\{|\xi|>\frac{1}{|h|}\}}\frac{\mu_{n}(d\xi)}{|\xi|^{q}}\nonumber \\
 & \leq4T |h|^{2-q}\sup_{n\geq1}\int_{\{|\xi|>1\}}\frac{\mu_{n}(d\xi)}{|\xi|^{q}}\nonumber \\
 & \leq C |h|^{2-q}.\label{eq:2}
\end{align}
Note that we have also used that $|h|\in(0,1)$.
Putting together \eqref{eq:111} and \eqref{eq:2}, we get that there
is a constant $C$ such that 
\begin{equation}
\sup_{n\geq1}\E{|v^{n}(t,x)-v^{n}(t,z)|^{2}}\leq C|x-z|^{2-q},\label{eq:6}
\end{equation}
for every $t\in[0,T]$ and $x,z\in\R^d$ such that $|x-z|<1$. This
estimate can be extended to any $x,z$ belonging to an arbitrary compact
set of $\R^d$. In this case, the constant $C$ depends on the underlying
compact set.

Let us now estimate the square moment of the time increments of $u_{n}$.
Let $t\in[0,T]$, $x\in\R^d$ and $h>0$ such that $t+h\leq T$. We
assume that $h<1$. Then, 
\begin{equation}
\E{|v_{n}(t+h,x)-v_{n}(t,x)|^{2}}\leq C(A_{1}^{n}+A_{2}^{n}),\label{eq:8}
\end{equation}
where 
\[
A_{1}^{n}=\E{\left|\int_{t}^{t+h}\int_{\R^d}G_{t+h-s}(x-y)W^{n}(ds,dy)\right|^{2}},
\]
\[
A_{2}^{n}=\E{\left|\int_{0}^{t}\int_{\R^d}\big\{ G_{t+h-s}(x-y)-G_{t-s}(x-y)\big\} W^{n}(ds,dy)\right|^{2}}.
\]
First, we deal with the term $A_{1}^{n}$. It clearly holds that 
\begin{align*}
A_{1}^{n} & =\int_{t}^{t+h}\int_{\R^d}|\F G_{t+h-s}(x-\cdot)(\xi)|^{2}\mu_{n}(d\xi)ds\\
 & =\int_{0}^{h}\int_{\R^d}|\F G_{s}(\xi)|^{2}\mu_{n}(d\xi)ds\\
 & =\int_{0}^{h}\int_{\R^d}\frac{\sin^{2}(s|\xi|)}{|\xi|^{2}}\mu_{n}(d\xi)ds.
\end{align*}
We have that, applying (a) in Lemma \ref{lem:1}, 
\begin{align*}
\int_{0}^{h}\int_{\{|\xi|\leq1\}}\frac{\sin^{2}(s|\xi|)}{|\xi|^{2}}\mu_{n}(d\xi)ds & \leq\int_{0}^{h}\int_{\{|\xi|\leq1\}}s^{2}\mu_{n}(d\xi)ds\\
 & \leq Ch^{3}\sup_{n\geq1}\mu_{n}(B_{1})\\
 & \leq Ch^{3}.
\end{align*}
On the other hand, by Hypothesis \textbf{(H1)},
we get 
\begin{align*}
\int_{0}^{h}\int_{\{|\xi|>1\}}\frac{\sin^{2}(s|\xi|)}{|\xi|^{2}}\mu_{n}(d\xi)ds & \leq\int_{0}^{h}\int_{\{|\xi|>1\}}\frac{\mu_{n}(d\xi)}{|\xi|^{2}}ds\\
 & \leq Ch\int_{\{|\xi|>1\}}\frac{\mu_{n}(d\xi)}{1+|\xi|^{2}}\\
 & \leq Ch\sup_{n\geq1}\int_{\R^d}\frac{\mu_{n}(d\xi)}{1+|\xi|^{2}}\\
 & \leq Ch,
\end{align*}
where we have used that
\[
\int_{\R^d}\frac{\mu_{n}(d\xi)}{1+|\xi|^2}
\leq C \int_{\R^d}\frac{\mu_{n}(d\xi)}{1+|\xi|^q}.
\] 
Hence, we have proved that 
\begin{equation}
\sup_{n\geq1}A_{1}^{n}\leq Ch.\label{eq:4}
\end{equation}
Regarding $A_{2}^{n}$, we have 
\begin{align}
A_{2}^{n} & =\int_{0}^{t}\int_{\R^d}|\F G_{t+h-s}(x-\cdot)(\xi)-\F G_{t-s}(x-\cdot)(\xi)|^{2}\mu_{n}(d\xi)ds\nonumber \\
 & =\int_{0}^{t}\int_{\R^d}\frac{1}{|\xi|^{2}}\big|\sin((t+h-s)|\xi|)-\sin((t-s)|\xi|)\big|^{2}\mu_{n}(d\xi)ds\nonumber \\
 & \leq C\int_{\R^d}\frac{1}{|\xi|^{2}}\min(1,h|\xi|)^{2}\mu_{n}(d\xi)\nonumber \\
 & =C\int_{\{|\xi|\leq\frac{1}{h}\}}h^{2}\mu_{n}(d\xi)+C\int_{\{|\xi|>\frac{1}{h}\}}\frac{\mu_{n}(d\xi)}{|\xi|^{2}}\nonumber \\
 & \leq Ch^{2}\sup_{n\geq1}\mu_{n}(B_{1/h})+Ch^{2-q}\sup_{n\geq1}\int_{\{|\xi|>1\}}\frac{\mu_{n}(d\xi)}{|\xi|^q}\nonumber \\
 & \leq Ch^{2-q}.\label{eq:5}
\end{align}
where we have applied Lemma \ref{lem:1}
and the fact that $h<1$. Estimates \eqref{eq:4} and \eqref{eq:5}
imply that there exists a constant $C$ such that 
\begin{equation*}
\sup_{n\geq1}\E{|v^{n}(t+h,x)-v^{n}(t,x)|^{2}}\leq Ch^{\min(1,2-q)},
\end{equation*}
for all $t\in[0,T]$, $x\in\R^d$ and $h\in(0,1)$ such that $t+h\leq T$.
This bound can be easily extended to all $h$ satisfying $t+h\leq T$. Moreover, by Remark \ref{rmk:5}, without any loose of generality we may assume that 
 $2-q<1$. Hence, it holds that
\begin{equation}
	\sup_{n\geq1}\E{|v^{n}(t+h,x)-v^{n}(t,x)|^{2}}\leq Ch^{2-q},\label{eq:12}
\end{equation}
for all $t\in[0,T]$, $x\in\R^d$ and any $h>0$ such that $t+h\leq T$.

Estimate \eqref{eq:12}, together with \eqref{eq:6}, allows us to
invoke Theorem \ref{thm: criterion} so that we deduce that the laws of
$\{v^{n},\,n\geq1\}$ are tight in the space $\mathcal{C}([0,T]\times\R^d)$.
Precisely, note that condition (i) of Theorem \ref{thm: criterion} is clearly
satisfied because $v^{n}(0,0)=0$. As far as condition (ii) is concerned,
recall that $v^n$ is a centered Gaussian process, and we have 
\[
\sup_{n\geq1}\E{|v^{n}(t',x')-v^{n}(t,x)|^{p}}\leq C\big(|t'-t|+|x'-x|\big)^{\delta},
\]
for all $p\geq2$, $t',t\in[0,T]$, $x',x\in J$ and any compact
$J\subset\R^d$, where 
\[
\delta=\frac{p}{2}(2-q).
\]
Thus, it suffices to take $p$ sufficiently large to ensure that (ii)
of Theorem \ref{thm: criterion} is fulfilled. This concludes the proof
of Proposition \ref{prop:tightness} in the case of the wave equation.
\qed


\subsubsection{Heat equation}

\label{sec: tightness-heat}

We now prove Proposition \ref{prop:tightness} in the case where $G$ in \eqref{eq: stochconv} is given by the fundamental solution of the heat equation in $\R^d$. It is well-known that
\[
\F G_{t}(\xi)=e^{-\frac{t|\xi|^{2}}{2}},\quad t>0,\;\xi\in\R^d.
\]
Let $t\in(0,T]$ and $x,z\in\R^d$, define $h:=z-x$ and assume that
$|h|\in(0,1)$. Then, arguing as in the case of the wave equation and
applying Fubini theorem and Lemma \ref{lem:1}, we have 
\begin{align}
\E{|v^{n}(t,x)-v^{n}(t,z)|^{2}} & \leq2\int_{0}^{T}\int_{\R^d}\big(1-\cos(<\xi, h>)\big)e^{-s|\xi|^{2}}\mu_{n}(d\xi)ds\nonumber \\
 & =2\int_{\R^d}\big(1-\cos(<\xi, h>)\big)\frac{1-e^{-T|\xi|^{2}}}{|\xi|^{2}}\mu_{n}(d\xi)\nonumber \\
 & \leq |h|^{2}\sup_{n\geq1}\mu_{n}(B_{1/|h|})+|h|^{2-q}\sup_{n\geq1}\int_{\{|\xi|>1\}}\frac{\mu_{n}(d\xi)}{|\xi|^q}\nonumber \\
 & \leq C |h|^{2-q}.\label{eq:11}
\end{align}
For the time increments, we argue as in the case of the wave equation
and consider the decomposition \eqref{eq:8}. Then, by Lemma \ref{lem:1},
\begin{align}
A_{1}^{n} & =\int_{t}^{t+h}\int_{\R^d}|\F G_{t+h-s}(x-\cdot)(\xi)|^{2}\mu_{n}(d\xi)ds\nonumber \\
 & =\int_{0}^{h}\int_{\R^d}e^{-s|\xi|^{2}}\mu_{n}(d\xi)ds\nonumber \\
 & =\int_{\R^d}\frac{1-e^{-h|\xi|^{2}}}{|\xi|^{2}}\mu_{n}(d\xi)\nonumber \\
 & =\int_{\{|\xi|^2\leq\frac{1}{h}\}}\frac{1-e^{-h|\xi|^{2}}}{|\xi|^{2}}\mu_{n}(d\xi)+\int_{\{|\xi|^2>\frac{1}{h}\}}\frac{1-e^{-h|\xi|^{2}}}{|\xi|^{2}}\mu_{n}(d\xi)\nonumber \\
 & \leq h\sup_{n\geq1}\mu_{n}\big(B_{\frac{1}{\sqrt{h}}}\big)+h^{1-\frac{q}{2}}\sup_{n\geq1}\int_{\{|\xi|>1\}}\frac{\mu_{n}(d\xi)}{|\xi|^{q}}\nonumber \\
 & \leq C h^{1-\frac{q}{2}}.
\label{eq:9}
\end{align}
On the other hand, we can argue as follows: 
\begin{align}
A_{2}^{n} & =\int_{0}^{t}\int_{\R^d}|\F G_{t+h-s}(x-\cdot)(\xi)-\F G_{t-s}(x-\cdot)(\xi)|^{2}\mu_{n}(d\xi)ds\nonumber \\
 & =\int_{0}^{t}\int_{\R^d}e^{-s|\xi|^{2}}\left(1-e^{-\frac{h|\xi|^{2}}{2}}\right)^{2}\mu_{n}(d\xi)ds\nonumber \\
 & =\int_{\R^d}\frac{1-e^{-t|\xi|^{2}}}{|\xi|^{2}}\left(1-e^{-\frac{h|\xi|^{2}}{2}}\right)^{2}\mu_{n}(d\xi)\nonumber \\
 & \leq\int_{\R^d}\frac{1}{|\xi|^{2}}\left(1-e^{-\frac{h|\xi|^{2}}{2}}\right)^{2}\mu_{n}(d\xi)\nonumber \\
 & \leq\int_{\{|\xi|^{2}\leq\frac{1}{h}\}}\frac{1}{|\xi|^{2}}\left(1-e^{-\frac{h|\xi|^{2}}{2}}\right)^{2}\mu_{n}(d\xi)+\int_{\{|\xi|^{2}>\frac{1}{h}\}}\frac{1}{|\xi|^{2}}\left(1-e^{-\frac{h|\xi|^{2}}{2}}\right)^{2}\mu_{n}(d\xi)\nonumber \\
 & \leq\frac{h^{2}}{4}\int_{\{|\xi|^{2}\leq\frac{1}{h}\}}|\xi|^{2}\mu_{n}(d\xi)+\int_{\{|\xi|^{2}>\frac{1}{h}\}}\frac{\mu_{n}(d\xi)}{|\xi|^{2}}\nonumber \\
 & \leq\frac{h}{4} \sup_{n\geq1} \mu_n\big(B_{\frac{1}{\sqrt{h}}}\big)  +h^{1-\frac{q}{2}}\int_{\{|\xi|^{2}>\frac{1}{h}\}}\frac{1}{|\xi|^{q}}\mu_{n}(d\xi)\nonumber \\
 & \leq\frac{h}{4} \sup_{n\geq1} \mu_n\big(B_{\frac{1}{\sqrt{h}}}\big)+h^{1-\frac{q}{2}}\int_{\{|\xi|>1\}}\frac{1}{|\xi|^{q}}\mu_{n}(d\xi)\nonumber \\
 & \leq C h^{1-\frac{q}{2}},
 \label{eq:10}
\end{align}
where we have also applied Lemma \ref{lem:1}.
Putting together estimates \eqref{eq:9} and \eqref{eq:10}, we end up with
\[
\sup_{n\geq 1} \E{|v^n(t+h,x)-v^n(t,x)|^2} \leq C h^{1-\frac{q}{2}}.
\]
Hence, owing to \eqref{eq:11}, we can conclude the proof as in the previous section. \qed


\subsection{Convergence of the covariance function}
\label{sec: cov}

We remind  that Proposition \ref{prop:tightness} states that the family of laws of 
$\{v^n,\, n\geq 1\}$ is tight in $\mathcal{C}([0,T]\times \R^d)$, and thus relatively compact in this space. The present section is devoted to identify the limit law by showing that the finite dimensional distributions of $v^n$ converge to those of $v$, where we recall that the latter is the Gaussian random field given by
\beq
v(t,x)=\int_0^t\int_{\R^d} G_{t-s}(x-y)W(ds,dy), \quad (t,x)\in [0,T]\times \R^d,
\label{eq:823}
\eeq
and here $W$ denotes a Gaussian spatially homogeneous noise as \eqref{eq:19} with spectral measure $\mu$. 
Since $\mu$ satisfies Dalang's condition (see Remark \ref{rmk: Dalang}), 
 the computations in Examples 6 and 8 of \cite{DalangEJP} allow us to conclude that $v$ is well-defined and satisfies, for all $p\geq 1$, 
\[
\sup_{(t,x)\in [0,T]\times \R^d} \E{|v(t,x)|^p}<\infty.
\]
In the next proposition, which is the main result of the present section, we show that 
the covariance function of $v^n$ converges to that of $v$, as $n\to\infty$. This fact has an important consequence. Namely, it implies the convergence of the corresponding finite dimensional distributions, because $v^n$, $n\geq 1$, and $v$ are centered Gaussian processes.

\begin{proposition}\label{prop:cov}
		Let $v^n$ and $v$ be the random fields defined by \eqref{eq: stochconv} and \eqref{eq:823}, respectively, where $G$ is the fundamental solution of the wave equation (resp. heat equation).
	Assume that Hypothesis \textbf{(H2)} is satisfied. Then, for all 
	$t,t'\in[0,T]$ and $x,x'\in\R^d$, it holds
	\[
	\lim_{n\to\infty} \E{v^{n}(t,x)v^{n}(t',x')} = \E{v(t,x)v(t',x')}.
	\]  
\end{proposition} 

\begin{proof}
Let us first deal with the case of the heat equation. 
Fix $t,t'\in[0,T]$ and $x,x'\in\R^d$. We may assume that $0\le t<t'$.
It holds  
\[
\E{v^{n}(t,x)v^{n}(t',x')}=\intt \int_{\R^d}  e^{-i<\xi,x-x'>}e^{-\frac{(t-s)}{2}|\xi|^{2}}e^{-\frac{(t'-s)}{2}|\xi|^{2}} \mu_{n}(d\xi)ds.
\]
We will see that this expression converges to 
\[
\E{v(t,x)v(t',x')}=\intt \int_{\R^d}  e^{-i<\xi,x-x'>}e^{-\frac{(t-s)}{2}|\xi|^{2}}e^{-\frac{(t'-s)}{2}|\xi|^{2}} \mu(d\xi)ds,
\]
as $n\to\infty$. Due to Hypothesis \textbf{(H2)}, and since $e^{-i<\xi,x-x'>}$ 
is bounded and continuous as a function of $\xi$, it suffices to
see that 
\[
I(\xi):=\intt e^{-\frac{(t-s)}{2}|\xi|^{2}}e^{-\frac{(t'-s)}{2}|\xi|^{2}}ds
\]
defines a continuous functions such that 
\beq
I(\xi)\le\frac{C_{t,t'}}{1+\axi^{2}},
\label{eq:21} 
\eeq
for all $\xi\in \R^d$, where $C_{t,t'}$ is some positive constant only depending on $t$
and $t'$.
By the dominated convergence theorem, it is clear that $I$ is a
continuous function. On the other hand, 
\[
I(\xi)=e^{-(t+t')\frac{|\xi|^{2}}{2}}\intt e^{s|\xi|^{2}}\,ds=\frac{1}{|\xi|^{2}}\big(e^{-(t'-t)\frac{|\xi|^{2}}{2}}-e^{-(t'+t)\frac{|\xi|^{2}}{2}}\big).
\]
We study separately the cases $\axi\le1$ and $\axi>1$. If $\axi\le1$,
by the mean value theorem, 
\[
e^{-(t'-t)\frac{|\xi|^{2}}{2}}-e^{-(t'+t)\frac{|\xi|^{2}}{2}}\le C_{t,t'}|\xi|^{2},
\]
and this implies that $I(\xi)\le C_{t,t'}.$ If $\axi>1$, we have
the obvious bound $I(\xi)\le1/|\xi|^{2}$. The above two facts imply
\eqref{eq:21}, which concludes the proof for the heat equation. 

Let us now prove Proposition \ref{prop:cov} in the case of the wave equation. Fix $t,t'\in[0,T]$ with $0\le t<t'$ and $x,x'\in\R^d$. We have that  
\[
\E{v^{n}(t,x)v^{n}(t',x')}= \intt \int_{\R^d} e^{-i<\xi,x-x'>}\frac{\sin((t-s)\axi)\sin((t'-s)\axi)}{\axi^{2}}\mu_{n}(d\xi)ds.
\]
As for the heat equation, it suffices to show that the function $J$
defined as 
\[
J(\xi)=\intt\frac{\sin((t-s)\axi)\sin((t'-s)\axi)}{\axi^{2}}ds, \quad \xi\in \R^d, 
\]
is continuous and satisfies 
\[
|J(\xi)|\le\frac{C_{t,t'}}{1+\axi^{2}},\quad \xi\in \R^d.
\]
First, we study the continuity of $J$. In the case
$0<\axi\le1$, we have 
\begin{equation}
\Big|\frac{\sin((t-s)\axi)\sin((t'-s)\axi)}{\xi^{2}}\Big|\le(t-s)(t'-s).\label{aco}
\end{equation}
The right-hand side of the above inequality, as a function of $s$,
belongs to $L^{1}([0,t])$. Hence, by the dominated convergence theorem, we have that
$J$ is continuous for $0<\axi\le1$, because the integrand in the
expression of $J$ is continuous. Secondly, if $\axi>1$, we have 
\[
\Big|\frac{\sin((t-s)\axi)\sin((t'-s)\axi)}{\xi^{2}}\Big|\le\frac{1}{\axi^{2}}.
\]
As before, applying again the dominated convergence theorem, we obtain the continuity of $J$ for $\axi>1$. Finally, we also need to consider the case $\xi=0$. Here, $J(0)$ is in principle not well-defined, so 
we must prove that $\lim_{\xi\to0}J(\xi)$ exists. To this end, we first note that, if $\xi$
belongs to a neighborhood of $0$, the estimate (\ref{aco}) is clearly satisfied. Next, we have that  
\begin{align*}
\lim_{\xi\to0}\frac{\sin((t-s)\axi)\sin((t'-s)\axi)}{\axi^{2}} & =\lim_{h\to0+}\frac{\sin((t-s)h)\sin((t'-s)h)}{h^{2}}\\
& =\lim_{h\to0+}\frac{((t'-s)h+o(h))((t-s)h+o(h))}{h^{2}}\\
& =(t'-s)(t-s).
\end{align*}
Therefore, applying the dominated convergence theorem, we obtain that 
\[
\lim_{\xi \to 0} J(\xi) = \intt(t'-s)(t-s)ds.
\]
It remains to prove that 
\beq
J(\xi)\le\frac{C_{t,t'}}{1+\axi^{2}},\quad \xi\in \R^d. 
\label{eq:23}
\eeq
If $\axi\le1$, It holds  
\[
|J(\xi)|=\left|\intt\frac{\sin((t-s)\axi)\sin((t'-s)\axi)}{\axi^{2}}\,ds\right|\le\intt(t'-s)(t-s)\,ds=C_{t,t}.
\]
If $\axi>1$, it is clear that 
\[
|J(\xi)|\le\intt\frac{1}{\axi^{2}}ds=\frac{t}{\axi^{2}}.
\]
Thus, we have verified \eqref{eq:23} and the proof of 
Proposition \ref{prop:cov} is now complete. 
\end{proof}


\section{Quasi-linear case: well-posedness and path continuity}
\label{sec: exis-cont}

This section is devoted to prove that equation \eqref{eq: mild} admits a unique solution which has a version with jointly continuous paths. The following result deals with the existence and uniqueness of solution to equation \eqref{eq: mild}. 

\begin{theorem}\label{thm: existence}
	Let $n\geq 1$ and $p\geq 2$. Assume that the initial data satisfy Hypothesis \ref{hyp: ic}, $b$ is a globally Lipschitz function and that Dalang's condition holds for the spectral measure $\mu_n$:
	\beq
	\int_{\R^d} \frac{\mu_n(d\xi)}{1+|\xi|^2} <\infty.
	\label{eq:27} 
	\eeq
	Then, equation \eqref{eq: mild} admits a unique solution in the space of $L^2(\Omega)$-continuous and adapted processes satisfying 
	\[
	\sup_{(t,x)\in [0,T]\times \R^d} \E{|u^n(t,x)|^p}<\infty. 
	\]   
\end{theorem}

\begin{proof}
It follows similar steps to those of \cite[Thm. 6]{DalangEJP} and
\cite[Thm. 3.1]{GJQ-Bernoulli} (see also  \cite[Thm. 4.3]{Dalang-Quer}). Indeed, it is important to remark that references \cite{DalangEJP} and \cite{Dalang-Quer} suppose that the corresponding noise's spectral measure is the inverse Fourier transform of a certain tempered measure, which we do not assume in the present paper (for example, in order to be able to treat fractional noises with $H<\frac12$). Nevertheless, the fact that we are dealing with an additive noise makes things easier for us and, in this sense, we can follow the same lines as in \cite[Thm. 3.1]{GJQ-Bernoulli}. We will mostly sketch the main steps to follow. 

We define the following Picard iteration scheme:
\[
u_0^n(t,x):= I_0^d(t,x) + \int_0^t \int_{\R^d} G_{t-s}(x-y)W^n(ds,dy), \quad (t,x)\in [0,T]\times \R^d,
\] 	
and, for $k\geq 1$,
\begin{equation}\label{eq:48}
u^n_k(t,x):=u_0^n(t,x) + \int_0^t  \big(b(u^n_{k-1}(s)) * G_{t-s}\big)(x) ds, \quad (t,x)\in [0,T]\times \R^d.
\end{equation}
Applying an induction argument one proves that, for all $k\geq 0$, the random field $u^n_k$ is adapted, $L^2(\Omega)$-continuous (thus has a jointly measurable modification) and satisfies 
\beq
\sup_{(t,x)\in [0,T]\times \R^d} \E{|u_k^n(t,x)|^p}<\infty. 
\label{eq:30}
\eeq
We will write the proof that $u^n_k$ is $L^2(\Omega)$-continuous, for all $k\geq 0$. 

First, let us verify that $u_0^n$ is $L^2(\Omega)$-continuous. The computations start as those in Sections \ref{sec: tightness-wave} and \ref{sec: tightness-heat}, but we point out that here, instead of hypotheses {\bf (H1)} and {\bf (H2)}, the spectral measure $\mu_n$ only satisfies Dalang's condition \eqref{eq:27}, so our strategy is slightly different.
First, we tackle the time increments. Let $(t,x)\in [0,T]\times \R^d$ and $h>0$ such that $t+h\leq T$. We consider the decomposition
\[
\E{|u_0^n(t+h,x)-u_0^n(t,x)|^2}\leq 2 (B_1+B_2+B_3),
\]
where 
\begin{align*}
B_1 & = |I_0^d(t+h,x)-I_0^d(t,x)|^2,\\
B_2 & = \E{\left|\int_{t}^{t+h}\int_{\R^d}G_{t+h-s}(x-y)W^{n}(ds,dy)\right|^{2}},\\
B_3 & = \E{\left|\int_{0}^{t}\int_{\R^d}\big\{ G_{t+h-s}(x-y)-G_{t-s}(x-y)\big\} W^{n}(ds,dy)\right|^{2}}.
\end{align*}
We will write the explicit computations in the case of the wave equation. The case of the heat equation can be done analogously. We know that $(t,x)\mapsto I_0^d(t,x)$ is continuous. Hence, for any compact $K\subset \R^d$, it holds
\[
\lim_{h\to 0} \sup_{x\in K} |I_0^d(t+h,x)-I_0^d(t,x)| =0.
\]
Next, we note that the term $B_2$ coincides with $A^n_1$ of Section \ref{sec: tightness-wave}. There, we proved that $B_2\leq C h$, uniformly in $(t,x)\in [0,T]\times \R^d$. Regarding $B_3$, it holds
\begin{align*}
	B_3 & = \int_0^t\int_{\R^d}\frac{1}{|\xi|^{2}}\big|\sin((s+h)|\xi|)-\sin(s|\xi|)\big|^{2}\mu_{n}(d\xi)ds\\
	& \leq T h^2  \int_{\{|\xi|\leq 1\}} \mu_n(d\xi) + 
	\int_0^T\int_{\{|\xi|>1\}} \frac{1}{|\xi|^{2}}\big|\sin((s+h)|\xi|)-\sin(s|\xi|)\big|^{2}\mu_{n}(d\xi)ds.	
\end{align*}
The first term in the right-hand side above clearly converges to $0$ as $h\to 0$; recall that $\mu_n(K)<\infty$ for any compact $K\subset \R^d$. As far as the second term is concerned, one applies the dominated convergence theorem and Dalang's condition on $\mu_n$ to deduce that it also converges to $0$ as $h\to0$. Both convergences hold uniformly with respect to $(t,x)\in [0,T]\times \R^d$. 

We now consider the spatial increments of $u_0^n$. Let $t\in [0,T]$ and $x,z\in \R^d$. We have
\[
\E{|u_0^n(t,x)-u_0^n(t,z)|^2}\leq 2 (C_1+C_2),
\]
where 
\begin{align*}
	C_1 & = |I^d_0(t,x)-I^d_0(t,z)|^2, \\
	C_2 & = \E{\left|\int_{0}^{t}\int_{\R^d}\big\{ G_{t-s}(x-y)-G_{t-s}(z-y)\big\} W^{n}(ds,dy)\right|^{2}}.
\end{align*}
As we did in Section \ref{sec: tightness-wave}, it holds that
\begin{align*}
	C_2 & = 2\int_0^t\int_{\R^d} \big(1-\cos(<\xi, x-z>)\big) 
	\frac{\sin^2\big((t-s)|\xi|\big)}{|\xi|^2}\mu_n(d\xi)ds\\
	& \leq  2\int_0^T\int_{\R^d} \big(1-\cos(<\xi, x-z>)\big) 
	\frac{\sin^2(s|\xi|)}{|\xi|^2}\mu_n(d\xi)ds\\
	& \leq \frac23 T^3 |x-z|^2 \int_{\{|\xi|\leq 1\}} \mu_n(d\xi) +
	2 T \int_{\{|\xi|>1\}} \frac{\big(1-\cos(<\xi, x-z>)\big)}{|\xi|^2} \mu_n(d\xi).	
\end{align*}
Both terms on the right-hand side above converge to $0$ as $|x-z|\to 0$, uniformly in $t\in [0,T]$. Thus, since $I^d_0$ is continuous, we have that, for any fixed $t\in [0,T]$, the map $x\mapsto u^n_0(t,x)$ is $L^2(\Omega)$-continuous. Then, we can argue as follows:
\begin{align*}
	& \limsup_{(s,y)\to(t,x)} \E{|u^n_0(s,y)-u^n_0(t,x)|^2} \\
	& \qquad \leq 
	C \limsup_{(s,y)\to(t,x)} \E{|u^n_0(s,y)-u^n_0(t,y)|^2} +
	C \limsup_{(s,y)\to(t,x)} \E{|u^n_0(t,y)-u^n_0(t,x)|^2}\\
	& \qquad \leq C \lim_{s\to t} \left( \sup_{y\in \R} \E{|u^n_0(s,y)-u^n_0(t,y)|^2}\right)
	+ C \lim_{y\to x} \E{|u^n_0(t,y)-u^n_0(t,x)|^2} .	
\end{align*}
As we proved above, the two latter limits vanish and we can conclude that $u^n_0$ is $L^2(\Omega)$-continuous. 

At this point, we assume that $u_k^n$ is $L^2(\Omega)$-continuous and let us check that $u_{k+1}^n$ satisfies the same property. The computations below work for both heat and wave equations. Using the usual notations, we first have that
\[
\E{|u_{k+1}^n(t+h,x)-u_{k+1}^n(t,x)|^2}\leq 2 (D_1+D_2+D_3),
\]
where
\begin{align*}
	D_1 & = \E{|u_0^n(t+h,x)-u_0^n(t,x)|^2},\\
	D_2 & = \E{\left|\int_0^t  \int_{\R^d} \big\{ b(u_k^n(t+h-s,x-y)) 
		- b(u_k^n(t-s,x-y))\big\} G_s(dy)ds\right|^2},\\
	D_3& =  \E{\left|\int_t^{t+h}  \int_{\R^d}  b(u_k^n(t+h-s,x-y)) 
		 G_s(dy)ds\right|^2}.	
\end{align*}
Let $K\subset \R^d$ be any compact set. We already proved that the term $D_1$ tends to $0$ as $h\to 0$, uniformly in $x\in K$. Using \eqref{eq:30}, one can easily prove that $D_3\leq C h$. Regarding $D_2$, we have that 
\[
D_2\leq C \int_0^t\int_{\R^d} \E{|u_k^n(t+h-s,x-y)- u_k^n(t-s,x-y)|^2} G_s(dy)ds.
\]
We will prove that, for any $\varepsilon>0$, there exists $\delta>0$ such that, for all $h\in (0,\delta)$, 
\[
\sup_{x\in K} \int_0^t\int_{\R^d} \E{|u_k^n(t+h-s,x-y)- u_k^n(t-s,x-y)|^2} G_s(dy)ds <\varepsilon.
\]
Let 
\[
B_k:=\sup_{(r,z)\in [0,T]\times \R^d} \E{|u^n_k(r,z)|^2},
\]
which we know, by the induction hypothesis, that it is a finite quantity. Fixed an arbitrary $\varepsilon>0$, we take a compact set $J\subset \R^d$ satisfying
\[
\int_0^T \int_{J^c} G_s(dy)\leq \frac{\varepsilon}{4 B_k}.
\]
Again by the induction hypothesis, we know that $u_k^n$ is uniformly $L^2(\Omega)$-continuous on compact sets. Then, there exists $\delta>0$ such that, if $h\in (0,\delta)$, 
\[
\sup_{\stackrel{(r,y)\in[0,T]\times J}{x\in K}}\E{|u^n_{k}(r+h,x-y)-u^n_{k}(r,x-y)|^2} \leq 
\frac{\varepsilon}{2\int_0^T \int_{\R^d}G_s(dy)ds}.
\]
Thus,
\begin{align*}
&	\int_0^t\int_{\R^d} \E{|u_k^n(t+h-s,x-y)- u_k^n(t-s,x-y)|^2} G_s(dy)ds \\
& \qquad \leq \int_0^T\int_J \E{|u_k^n(t+h-s,x-y)- u_k^n(t-s,x-y)|^2} G_s(dy)ds
 + 2 B_k \int_0^T \int_{J^c}  G_s(dy)ds \\
 & \qquad \leq \varepsilon.
\end{align*}	
Hence, we conclude that $t\mapsto u^n_{k+1}(t,x)$ is $L^2(\Omega)$-equicontinuous for $x\in K$. 	

Let us now deal with the spatial increments of $u_{k+1}^n$. We have
\[
\E{|u_{k+1}^n(t,x)-u_{k+1}^n(t,z)|^2}\leq 2 (E_1+E_2),
\] 
where
\begin{align*}
	E_1 & = \E{|u_0^n(t,x)-u_0^n(t,z)|^2},\\
	E_2 & = \E{\left|\int_0^t  \int_{\R^d} \big\{ b(u_k^n(t-s,x-y)) 
		- b(u_k^n(t-s,z-y))\big\} G_s(dy)ds\right|^2}.
\end{align*}
The term $E_1$ converges to $0$ as $|x-z|\to 0$, because $u_0^n$ is $L^2(\Omega)$-continuous. On the other hand, it holds 
\[
E_2\leq C \int_0^t\int_{\R^d} \E{|u_k^n(t-s,x-y)- u_k^n(t-s,z-y)|^2} G_s(dy)ds.
\]
Here, we invoke again the induction hypothesis and the estimate \eqref{eq:30}, together with an application of the dominated convergence theorem. Therefore, $E_2$ tends to $0$ as $|x-z|\to 0$. We conclude that, for any fixed $t\in[0,T]$, the map $x\mapsto u^n_{k+1}(t,x)$ is $L^2(\Omega)$-continuous. Arguing as we did for $u^n_0$, we have that $u_k^n$ is  $L^2(\Omega)$-continuous. This implies that $u_k^n$ admits a jointly measurable version, which is clearly adapted. These facts, together with \eqref{eq:30}, let us conclude that $u^n_k$ is well-defined for all $k\geq 1$. 

Next step consists in proving that the Picard iteration scheme $\{u_k^n,\, k\geq 1\}$ converges in the space of $L^2(\Omega)$-continuous,
adapted and $L^p (\Omega)$-uniformly bounded processes, which is a complete normed space when endowed with the norm
\[
\|w\|_p:=\sup_{(t,x)\in [0,T]\times \R^d} \|w(t,x)\|_{L^p(\Omega)}.
\]
This can be done as Step 2 in the proof of \cite[Thm. 3.1]{GJQ-Bernoulli}. We denote by $\{u^n(t,x),\, (t,x)\in [0,T]\times \R^d\}$ the underlying limit. In particular, it holds that 
\[
\lim_{k\to\infty} \sup_{(t,x)\in [0,T]\times \R^d} \E{|u^n_k(t,x) - u^n(t,x)|^p}=0. 
\]
Since any Picard iterate $u^n_k$ is $L^2(\Omega)$-continuous and adapted, the limit $u^n$ has the same properties. In particular, it has a joint-measurable version, which will be denoted in the same way.

The final step consists in checking that $u^n$ is the solution of equation \eqref{eq: mild} and that it is unique. These statements can be proved using standard arguments. The proof is thus complete.	
\end{proof}


In the following subsections we will prove that the solutions of \eqref{eq: mild} and \eqref{eq: mild_u} have a modification with continuous sample paths. First, we will deal with the stochastic wave equation, next with the stochastic heat equation with bounded drift and, finally, with the stochastic heat equation with arbitrary drift coefficient. The reasons why we follow these steps are the following: 

We aim to show that the solutions of \eqref{eq: mild} and \eqref{eq: mild_u} admit a continuous modification under the minimal assumptions on the initial data. For the wave equation and the heat equation with bounded drift, those hypotheses are the same as for the existence and uniqueness of solution. The precise details will be given below, but let us reveal that our strategy is based on solving a certain deterministic equation (see \eqref{eq:69}). Moreover, as it be explained later on, this method will allow us to achieve, in a rather straightforward way, the convergence in law of our main result (Theorem \ref{thm: main}) for those cases. 

The case of the heat equation with arbitrary drift must be treated in a different way. This is because the above-mentioned deterministic equation is not well-posed for any Lipschitz-continuous drift. More precisely, the corresponding first-order Picard iterate contains the integral
\[
\int_0^t\int_{\R^d} G_{t-s}(x-y)b(\eta(s,y))dyds, 
\]
where $G_s(y)=(2\pi s)^{-\frac d2}e^{-\frac{|y|^2}{2s}}$ and $\eta\in \C([0,T]\times \R^d)$. This integral may not be well-defined.


\subsection{Wave equation}
\label{sec: wave}

This section is devoted to prove the following result.

\begin{theorem}\label{thm: cont}
Let $n\geq 1$ and consider $u^n$ the solution to \eqref{eq: waven}, which satisfies the mild form \eqref{eq: mild}, where the fundamental solution $G$ is given by \eqref{eq:1} and \eqref{eq: 3d}.
Assume that, for some $q\in (0,2)$ the spectral measure $\mu_n$ satisfies
 \begin{equation}
  \int_{\R^d} \frac{\mu_n(d\xi)}{1+|\xi|^q} <\infty.
 \label{eq:45}
\end{equation}
Assume that $b:\R\to\R$ is globally Lipschitz and  the initial data satisfy (ii) in Hypothesis \ref{hyp: ic}.
Then, the random field $u^n$ admits a modification with continuous sample paths.
\end{theorem}

\begin{remark}
 In Theorem \ref{thm: cont}, we need to slightly strengthen Dalang's condition on the spectral measure $\mu_n$. We also point out that the assumptions on the initial data are the same as in Theorem \ref{thm: existence}, where we showed existence and uniqueness of solution.
\end{remark}

\begin{remark}
 One could also assume more regularity on the initial data so that the underlying solution has a version with H\"older continuous paths. In this sense, we have decided to keep the assumptions on $u_0$ and $v_0$ as general as possible, because for our purposes we only need continuity of the corresponding sample paths.
\end{remark}

In the proof of Theorem \ref{thm: cont}, we will make use of the following ad-hoc version of Grönwall's lemma, which corresponds to the extension of \cite[Lem. 4.2]{GJQ-Bernoulli} to any space dimension $d\in\{1,2,3\}$. We give its proof for the sake of completeness.

\begin{lemma}\label{lem: gron-w}
 Let $\{f_k,\, k\geq 0\}$ be sequence of measurable and non-negative functions defined on $[0,T]\times B_{L+T}$, where $T, L>0$ and $B_{L+T}=\{y\in \R^d, |y|\leq L+T\}$. Assume that there exist $\lambda_1,\lambda_2>0$ such that, for all $(t,x)\in [0,T]\times B_L$ and $k\geq 0$,
 \beq\label{eq: 90}
  f_{k+1}(t,x)\leq \lambda_1 + \lambda_2 \int_0^t \big(f_k(s,\cdot)*G_{t-s}\big)(x)ds,
 \eeq
 where $G$ is the fundamental solution of the wave equation in $\R^d$, $d\in \{1,2,3\}$, 
and $f_0$ is bounded. Then, for all $k\geq 0$ and $(t,x)\in [0,T]\times B_L$, it holds
\beq\label{eq:91}
f_k(t,x)\leq \lambda_1 \sum_{j=0}^{k-1} \frac{(\lambda_2 t^2)^j}{j!} + 
\sup_{\overset{r\in [0,T]}{z\in B_{L+T}}} \hspace{-0.1cm} |f_0(r,z)| \, 
\frac{(\lambda_2 t^2)^k}{k!}.
\eeq
\end{lemma}

\begin{proof}
 We will apply an induction argument. For $k=1$, we need to verify that
 \[
  f_1(t,x)\leq \lambda_1 + \lambda_2 t^2 \|f_0\|_{T,L,\infty},
 \]
 where 
 \[
 \|f_0\|_{T,L,\infty} := \sup_{\overset{r\in [0,T]}{z\in B_{L+T}}} \hspace{-0.1cm} |f_0(r,z)|.
 \]
Note that it suffices to prove that, for all measurable and bounded function $f:\R^d\to\R_+$, it holds, for any fixed $t\in[0,T]$,
\beq
 \sup_{(s,x)\in [0,t]\times \R^d} (f*G_{t-s})(x)\leq t \|f\|_{T,L,\infty}.
\label{eq:92}
 \eeq
This property is straightforward for the case $d=1$. If $d=2$, we have, by \eqref{eq:56},
\[
 (f*G_{t-s})(x) \leq \|f\|_{T,L,\infty} \|G_{t-s}\|_{L^1(\R^2)}\leq t \|f\|_{T,L,\infty},
\]
for all $(s,x)\in [0,t]\times \R^2$.
Finally, for $d=3$, applying again \eqref{eq:56} we end up with
\[
 (f*G_{t-s})(x) =\int_{\R^3} f(x-y)G_{t-s}(dy)\leq \|f\|_{T,L,\infty}\,  t,
\]
for all $(s,x)\in [0,t]\times \R^3$. Hence, \eqref{eq:91} is valid for $k=1$. Next, assume that \eqref{eq:91} holds for some $k>1$. Then, applying \eqref{eq:92} and the induction hypothesis, one can argue as follows: for all $(t,x)\in [0,T]\times \R^d$,
\begin{align*}
 f_{k+1}(t,x) & \leq \lambda_1 + \lambda_2 \int_0^t \big(f_k(s,\cdot)*G_{t-s}\big)(x)ds \\
 & \leq \lambda_1 + \lambda_2 \int_0^t \left(\lambda_1 \sum_{j=0}^{k-1} \frac{(\lambda_2 s^2)^j}{j!}
+ \|f_0\|_{T,L,\infty}  \frac{(\lambda_2 s^2)^k}{k!}\right) t ds\\
& \leq \lambda_1 + \lambda_1 \sum_{j=0}^{k-1} \frac{\lambda_2^{j+1} t^{2j+2}}{(j+1)!} +
\|f_0\|_{T,L,\infty} \frac{\lambda_2^{k+1}t^{2k+2}}{(k+1)!}\\
& = \lambda_1 \sum_{j=0}^k \frac{(\lambda_2 t^2)^j}{j!} + \|f_0\|_{T,L,\infty} \frac{(\lambda_2t^2)^{k+1}}{(k+1)!}.
\end{align*}
Thus, \eqref{eq:91} holds for $k+1$ and the proof is complete.
\end{proof}

\noindent {\it Proof of Theorem \ref{thm: cont}.} It will be developed through several steps.

{\it Sept 1.} We recall that, by Lemma \ref{lem: dqs}, the function $(t,x)\mapsto I_0^d(t,x)$ is continuous (and uniformly bounded) on $[0,T]\times \R^d$. Next, we define, for any $(t,x)\in [0,T]\times \R^d$,
\beq
 v^n(t,x):= \int_{0}^{t}\int_{\R^d} G_{t-s}(x-y) W^n(ds,dy).
\label{eq:120}
\eeq
Applying similar arguments as those used in the proof of Proposition \ref{prop:tightness} (see Section \ref{sec: tightness-wave}), one proves that condition \eqref{eq:45} implies the following. There exists a constant $C_n>0$ such that, for all $x,z\in \R^d$, we have
\beq
 \sup_{t\in[0,T]} \E{|v^n(t,x)-v^n(t,z)|^2} \leq C_n
 |x-z|^{2-q}.
 \label{eq:95}
\eeq
Moreover, for any $s,t\in [0,T]$, we have
\beq
 \sup_{x\in \R^d} \E{|v^n(t,x)-v^n(s,x)|^2} \leq C_n
 |t-s|^{2-q}.
 \label{eq:96}
\eeq
We remark that, in Proposition \ref{prop:tightness}, we wanted the above estimates to be uniform with respect to $n$. That is the reason why we needed to assume the stronger assumption {\bf{(H1)}}.

Let us sketch the proof of \eqref{eq:95}. As in \eqref{eq:99}, we have
\[
\E{|v^{n}(t,x)-v^{n}(t,z)|^{2}} \leq2\int_{0}^{T}\int_{\R^d}\big(1-\cos(<\xi, h>)\big)\frac{\sin^{2}(s|\xi|)}{|\xi|^{2}}\mu_{n}(d\xi)ds,
\]
where $h=z-x$. On the one hand, the inequality $1-\cos(y)\leq\frac{y^{2}}{2}$, $y\in \R$, implies that
\begin{align}
 2\int_{0}^{T}\int_{\{|\xi|\leq 1\}}\big(1-\cos(<\xi, h>)\big)\frac{\sin^{2}(s|\xi|)}{|\xi|^{2}}\mu_{n}(d\xi)ds & \leq
 |h|^2 \int_{0}^{T}\int_{\{|\xi|\leq 1\}} \sin^{2}(s|\xi|)\mu_{n}(d\xi)ds\nonumber \\
 & \leq T |h|^2 \mu_n(\{|\xi|\leq 1\}) \nonumber \\
 & \leq C_n |h|^2.
 \label{eq:101}
\end{align}
In the latter estimate, we have used that $\mu_n$ is a tempered measure, which implies that any bounded set has finite measure. On the other hand, note that $1-\frac q2\in (0,1)$ and
\[
 1-\cos(<\xi, h>) \leq \big(1-\cos(<\xi, h>)\big)^{1-\frac q 2}.
\]
Hence, by \eqref{eq:45},
\begin{align}
 2\int_{0}^{T}\int_{\{|\xi|> 1\}}\big(1-\cos(<\xi, h>)\big)\frac{\sin^{2}(s|\xi|)}{|\xi|^{2}}\mu_{n}(d\xi)ds & \leq 2^{\frac q 2} T
 |h|^{2-q} \int_{\{|\xi| > 1\}}
 \frac{\mu_{n}(d\xi)}{|\xi|^q}
 ds \nonumber \\
 & \leq C |h|^{2-q} \int_{\{|\xi| > 1\}}
 \frac{\mu_{n}(d\xi)}{1+|\xi|^q}
  \nonumber \\
 & \leq  C |h|^{2-q} \int_{\R^d}
 \frac{\mu_{n}(d\xi)}{1+|\xi|^q}
  \nonumber \\
 &\leq C_n |h|^{2-q}.
 \label{eq:100}
\end{align}
Estimates \eqref{eq:101} and \eqref{eq:100} imply \eqref{eq:95}.
In order to prove \eqref{eq:96}, we assume that $t>s$ and observe that
\begin{equation}
\E{|v_{n}(t,x)-v_{n}(s,x)|^{2}}\leq C(A_1^n+A_2^n),\label{eq:888}
\end{equation}
where
\[
A_{1}^{n}=\E{\left|\int_{s}^{t}\int_{\R^d}G_{t-r}(x-y)W^{n}(dr,dy)\right|^{2}},
\]
\[
A_{2}^{n}=\E{\left|\int_{0}^s\int_{\R^d}\big\{ G_{t-r}(x-y)-G_{s-r}(x-y)\big\} W^{n}(dr,dy)\right|^{2}}.
\]
The term $A_1^n$ can be treated as in Section \ref{sec: tightness-wave}, yielding
\beq
 A_1^n\leq C\left( (t-s)^3 \mu_n(\{|\xi|\leq 1\})
 + (t-s)\int_{\R^d}
 \frac{\mu_n(d\xi)}{1+|\xi|^2}\right) \leq C_n (t-s).
\label{eq:103}
\eeq
In order to deal with $A_2^n$, we argue as follows, taking into account \eqref{eq:45} and that $\mu_n$ is tempered:
\begin{align*}
A_{2}^{n} & =\int_{0}^{s}\int_{\R^d}\big|\sin((t-r)|\xi|)-\sin((s-r)|\xi|)\big|^{2} \frac{\mu_{n}(d\xi)}{|\xi|^{2}}dr\nonumber \\
 & \leq T h^2 \mu_n(\{|\xi|\leq 1\})
 + 2^q\int_{0}^{s}\int_{\{|\xi|>1\}}\big|\sin((t-r)|\xi|)-\sin((s-r)|\xi|)\big|^{2-q} \frac{\mu_{n}(d\xi)}{|\xi|^{2}}dr\\
 & \leq C_n (t-s)^2 + 2^q T (t-s)^{2-q}
 \int_{\{|\xi|>1\}}\frac{\mu_{n}(d\xi)}{|\xi|^q}\\
 & \leq C_n (t-s)^{2-q}.
\end{align*}
This bound, together with \eqref{eq:103}, implies
\eqref{eq:96}, since we may assume, without loosing generality, that $2-q\leq 1$.
Finally, by Kolmogorov continuity criterion, estimates \eqref{eq:95} and \eqref{eq:96} imply that the random field $v^n$ has a version with jointly Hölder-continuous paths.

{\it Step 2.} Let $\eta\in \C([0,T]\times \R^d)$. This section is devoted to prove that the following (deterministic) integral equation has a unique solution in the space $\C([0,T]\times \R^d)$:
\beq
z(t,x)= \eta(t,x) + \int_0^t \big(b(z(s))*G_{t-s}\big)(x)ds,
\label{eq:69}
\eeq
for all $(t,x)\in[0,T]\times \R^d$. Here, we have used the notation $z(s):=z(s,\cdot)$. We recall that $b$ is Lipschitz-continuous and $G$ is the fundamental solution of the wave equation (see \eqref{eq:1} and \eqref{eq: 3d}). Next, we will show that the operator 
\beq
\begin{array}{cccc}
F: & \C([0,T]\times \R^d) & \rightarrow &\C([0,T]\times \R^d)	\\
   &  \eta & \mapsto & F(\eta)=z,
\end{array}
\label{eq:190}
\eeq
is continuous. The latter statement is not needed to conclude the proof of Theorem \ref{thm: cont}, but it will be crucial to show the validity of the main result of the paper (Theorem \ref{thm: main}) in the case of the wave equation.

The proof follows the same lines as that of \cite[Thm. 4.3]{GJQ-Bernoulli}. So we will only point out the main differences, which are due to the fact that we are dealing with any dimension $d\in\{1,2,3\}$.

We start by defining the corresponding Picard iteration scheme: for any $(t,x)\in [0,T]\times \R^d$, set
\begin{align*}
 z_0(t,x) & :=\eta(t,x),\\
 z_k(t,x) & := \eta(t,x) + \int_0^t \big(b(z_{k-1}(s))*G_{t-s}\big)(x)ds, \quad k\geq 1.
\end{align*}
One can easily verify that the above are well-defined random fields and, moreover, using an induction argument, $z_k$ is a continuous function, for all $k\geq 0$. Next, we show that, as $k\to\infty$, $z_k$ converges uniformly on compact sets on $[0, T ] \times \R^d$.

Let $(t,x)\in [0,T]\times B_L$, with $L>0$ is arbitrary, where we recall that $B_{L}=\{y\in \R^d, |y|\leq L\}$. Owing to the Lipschitz property of $b$, we have, for any $k\geq 1$,
\[
 |z_{k+1}(t,x)-z_k(t,x)|\leq C
 \int_0^t \int_{\R^d} \big( |z_k(s)-z_{k-1}(s)|*G_{t-s}\big)(x) ds.
\]
At this point, we take $f_k(t,x):=|z_{k+1}(t,x)-z_k(t,x)|$ and we apply Lemma \ref{lem: gron-w}. Thus, we deduce that the sequence $\{z_k(t,x)\}_{k\geq 0}$ is uniformly Cauchy on $\C([0,T]\times B_L)$.  The limit of this sequence is denoted by $z(t, x)$. The uniqueness of the point-wise limit, the fact that $\C([0, T ] \times \R^d)$ is a complete
metric space, with the topology of uniform convergence on compact sets, and the continuity of $z_k$, for all $k\geq 0$, imply that $z$  also defines a continuous function in $\C([0, T ] \times \R^d)$. Furthermore, one can easily verify that $z$ solves equation \eqref{eq:69}. Uniqueness can be showed by applying again Lemma \ref{lem: gron-w}.

As far as the continuity of the solution operator $F$ is concerned, it is straightforward to show that, for all $\eta_1,\eta_2\in \C([0,T]\times \R^d)$ and $(t,x)\in [0,T]\times B_L$,
\[
 |F(\eta_1)(t,x)-F(\eta_2)(t,x)|\leq
 \|\eta_1-\eta_2\|_{L,\infty} +
 C \int_0^t\int_{\R^d}
 \big(|F(\eta_1)(s)-F(\eta_2)(s)|*G_{t-s}\big)(x)ds,
\]
where $\|\cdot\|_{L,\infty}$ denotes the supreme norm on $\C([0,T]\times B_L)$. Then, again by Lemma \ref{lem: gron-w},
\[
 \|F(\eta_1)-F(\eta_2)\|_{L, \infty}\leq C \|\eta_1-\eta_2\|_{L,\infty}.
\]
This concludes Step 2.

{\it Step 3.} By Step 1, we know that the sample paths of $I_0^d + v^n$ are continuous, almost surely. Then, in equation \eqref{eq:69}, we take one of the continuous trajectories of the latter random field: 
\[
\eta(t,x)=I_0^d(t,x)+v^n(t,x),\quad (t,x)\in [0,T]\times \R^d.
\]
It is clear that the corresponding path of the solution $u^n$ to equation \eqref{eq: waven} is given by the solution $z$ to equation  \eqref{eq:69}. Hence, by Step 2, the paths of $u^n$ are almost sure continuous. This concludes the proof. 
\qed

\subsection{Heat equation with bounded drift}
\label{sec: heat-bounded}

The aim of this section is to prove the following:

\begin{theorem}\label{thm: cont-heat-bdd}
	Let $n\geq 1$ and consider $u^n$ the solution to \eqref{eq: heatn}, which satisfies the mild form \eqref{eq: mild}, where the fundamental solution $G$ is given by \eqref{eq: heat kernel}. We assume that $b:\R\to\R$ is globally Lipschitz and bounded and  $u_0$ satisfies (i) in Hypothesis \ref{hyp: ic}. Suppose that, for some $q\in (0,2)$ the spectral measure $\mu_n$ satisfies \eqref{eq:45}.
	Then, the random field $u^n$ admits a modification with continuous sample paths.
\end{theorem}

In the proof of Theorem \ref{thm: cont-heat-bdd}, we will need the following ad-hoc version of Gronwall's lemma, which is the analogous of Lemma \ref{lem: gron-w} adapted to the heat equation. Its proof follows exactly the same lines as that of the latter result, and therefore will be omitted. 

\begin{lemma}\label{lem: gron-h}
Let $\{f_k,\, k\geq 0\}$ be sequence of measurable functions defined on $[0,T]\times \R^d$. Assume that there exist $\lambda_1,\lambda_2>0$ such that, for all $(t,x)\in [0,T]\times \R^d$ and $k\geq 0$,
\[
|f_{k+1}(t,x)-f_k(t,x)| \leq \lambda_1 + \lambda_2 \int_0^t \big(  [b(f_k(s))-b(f_{k-1}(s))]*G_{t-s}\big)(x)ds,
\]
where $G$ is the fundamental solution of the heat equation in $\R^d$, $d\geq 1$, and  $b$ is a bounded and Lipschitz function, with Lipschitz constant $C_b$. Then, for all $(t,x)\in [0,T]\times \R^d$, 
\[
|f_{k+1}(t,x)-f_k(t,x)|  \leq 2\|b\|_\infty C_b^{k-1} 
\frac{(\lambda_2 t)^k}{k!} +\sum_{j=0}^{k-1} \frac{\lambda_1 t^j}{j!}.
\]
As a consequence, it holds
\[
\limsup_{k\to\infty} \bigg( \sup_{x\in \R^d} 
|f_{k+1}(t,x)-f_k(t,x)|  \bigg) \leq \lambda_1 e^t.
\]	
\end{lemma}

\noindent {\it Proof of Theorem \ref{thm: cont-heat-bdd}.} As in the proof of Theorem \ref{thm: cont}, first we point out that $I_0^d$ is continuous (by \cite[Lem. 4.2]{Dalang-Quer}). Next, we consider the random field $v^n$ defined as in \eqref{eq:120}, but with $G$ being the fundamental solution of the heat equation. Using similar arguments as those in Section \ref{sec: tightness-heat}, we check that, under condition \eqref{eq:45}, there exists $C_n>0$ such that, for all $x,z\in \R^d$, we have
\beq
\sup_{t\in[0,T]} \E{|v^n(t,x)-v^n(t,z)|^2} \leq C_n
|x-z|^{2-q}.
\label{eq:955}
\eeq
Moreover, for any $s,t\in [0,T]$, 
\beq
\sup_{x\in \R^d} \E{|v^n(t,x)-v^n(s,x)|^2} \leq C_n
|t-s|^{1-\frac q2}.
\label{eq:966}
\eeq
For the space increments, we have, by the computations that let to \eqref{eq:11} (and setting $h:=z-x$),
\begin{align*}
\E{|v^n(t,x)-v^n(t,z)|^2} & \leq 	
2\int_{\R^d}\big(1-\cos(<\xi, h>)\big)\frac{1-e^{-T|\xi|^{2}}}{|\xi|^{2}}\mu_{n}(d\xi)\\	
& \leq C \big( |h|^2 \mu_n(\{|\xi|\leq 1\}) + 
|h|^{2-q} \int_{\R^d} \frac{\mu_n(d\xi)}{1+|\xi|^q}	\big)\\
& \leq C_n |h|^{2-q}.
\end{align*}
For the time increments, we assume that $t>s$ and we consider decomposition \eqref{eq:888}. In order to deal with the term $A_1^n$, we apply that $1-e^{-y}\leq y$, $y\in \R_+$, and the fact that $1-\frac q2\in (0,1)$. Thus,
\begin{align*}
	A_1^n & = \int_{\R^d} \frac{1-e^{(t-s)|\xi|^2}}{|\xi|^2}\mu_n(d\xi) \\
	& \leq (t-s)\mu_n(\{|\xi|\leq 1\}) + \int_{\{|\xi|>1\}} 
	\frac{\big(1-e^{(t-s)|\xi|^2}\big)^{1-\frac q2}}{|\xi|^2}\mu_n(d\xi)\\
	& \leq 	C_n (t-s)^{1-\frac q2}.
\end{align*}
The term $A_2^n$ can be treated in the same way, yielding
$A_n^2 \leq C_n (t-s)^{1-\frac q2}$. Hence, estimates \eqref{eq:955} and \eqref{eq:966} hold true. 
By Kolmogorov continuity criterion, we can conclude that $v^n$ admits a version with jointly (Hölder-)continuous paths. 

The remaining of the proof follows as in Steps 2 and 3 of the proof of Theorem \ref{thm: cont}. More precisely, one considers \eqref{eq:69} with $G$ being the fundamental solution of the heat equation and assuming that $b$ is a bounded function. Then, using Lemma \ref{lem: gron-h}, one proves that equation \eqref{eq:69} admits a unique solution in the space $\C([0,T]\times \R^d)$ and, moreover, the operator $F$ defined in \eqref{eq:190} is continuous. Finally, one concludes the proof by taking, in equation \eqref{eq:69}, $\eta(t,x)=I^d_0(t,x)+ v^n(t,x)$, for $(t,x)\in [0,T]\times \R^d$.     
\qed

\subsection{Heat equation with general drift}
\label{sec: heat-general}

In this section, we will deal with the stochastic heat equation with a general globally Lipschitz drift $b$. Our aim is to prove the following result. Since we aim to apply Kolmogorov continuity criterion directly to the solution $u^n$ of \eqref{eq: heatn}, we are forced to assume more regularity on the initial condition. 

\begin{proposition}\label{prop: heat}
 Let $n\geq 1$  and consider $u^n$ the solution to \eqref{eq: heatn}, which satisfies the mild form \eqref{eq: mild}, where the fundamental solution $G$ is given by \eqref{eq: heat kernel}. We assume that $b:\R\to\R$ is globally Lipschitz and $u_0$ satisfies (i) in Hypothesis \ref{hyp: ic}. Moreover, suppose that $u_0\in \C^\alpha(\R^d)$, for some $\alpha\in (0,1)$. Assume that, for some $q\in (0,2)$ the spectral measure $\mu_n$ satisfies \eqref{eq:45}. Then, for any $p\geq 1$, 	
 there exists $C_n>0$ such that, for all $x,z\in \R^d$, we have
\beq
\sup_{t\in[0,T]} \E{|u^n(t,x)-u^n(t,z)|^p} \leq C_n
|x-z|^{p\beta},
\label{eq:975}
\eeq
where $\beta=\min(\alpha,1-\frac q2)$. 
Moreover, for any $s,t\in [0,T]$, we have
\beq
\sup_{x\in \R^d} \E{|u^n(t,x)-u^n(s,x)|^p} \leq C_n
|t-s|^{p\frac \beta 2}.
\label{eq:976}
\eeq	
As a consequence, $u^n$ admits a version with jointly Hölder-continuous paths. 
\end{proposition}

\begin{proof}
First, in the proof of \cite[Thm. 4.3]{Sanz-Sarra} it has been showed that, for all $x,z\in \R^d$, 
\beq
\sup_{t\in [0,T]} |I^d_0(t,x) - I^d_0(t,z)|\leq C |x-z|^\alpha, 
\label{eq:300}
\eeq
and for all $s,t\in [0,T]$,
\beq
\sup_{x\in \R^d} |I^d_0(t,x) - I^d_0(s,x)|\leq C |t-s|^{\frac \alpha 2}. 
\label{eq:301}
\eeq
Next we define, as in \eqref{eq:120}, 	
\[
v^n(t,x):=\int_0^t\int_{\R^d} G_{t-s}(x-y)W^n(ds,dy), \quad (t,x)\in [0,T]\times \R^d.
\]
The second-order moments of the space and time increments of $v^n$ have been studied in the proof of Theorem \ref{thm: cont-heat-bdd}; see estimates \eqref{eq:955} and \eqref{eq:966} therein. Then, since $v^n$ is a Gaussian random field, it holds, for all $p\geq 1$ and $x,z\in \R^d$,
\beq
\sup_{t\in[0,T]} \E{|v^n(t,x)-v^n(t,z)|^p} \leq C_n
|x-z|^{(1-\frac q2)p},
\label{eq:9955}
\eeq
and for any $s,t\in [0,T]$, 
\beq
\sup_{x\in \R^d} \E{|v^n(t,x)-v^n(s,x)|^p} \leq C_n
|t-s|^{(\frac 12-\frac q4)p}.
\label{eq:9966}
\eeq

From now on, we follow similar arguments as those used in the proof of \cite[Thm. 2.1]{Sanz-Sarra}, so we will only sketch the main computations. 
Let $x,z\in \R^d$ and $t\in [0,T]$, and denote $h:=z-x$. Taking into account \eqref{eq:300}, \eqref{eq:9955} and the Lipschitz assumption on $b$, and applying Hölder inequality with respect to the finite measure $G_{t-s}(y)dyds$ on $[0,t]\times\R^d$, one can readily check that, for all $p\geq 1$,	
\begin{align*}
&\E{|u^n(t,x+h) - u^n(t,x)|^p}\\ 
& \qquad  \leq C_n |h|^{\beta p}  	+
C \int_0^t \int_{\R^d} \E{|u^n(s,x + h-y) - u^n(s,x-y)|^p} G_{t-s}(y)dyds\\
&  \qquad = C_n |h|^{\beta p} 
+ C \int_0^t \int_{\R^d} \E{|u^n(s,y+h) - u^n(s,y)|^p}  G_{t-s}(x-y)dyds\\
& \qquad \leq C_n |h|^{\beta p} 
+ C \int_0^t  \sup_{y\in \R^d}\E{|u^n(s,y+h) - u^n(s,y)|^p} ds.
\end{align*}	
Note that we have used that $\int_{\R^d} G_{t-s}(x-y)dy=1$, for all $(t,x)\in [0,T]\times \R^d$.
Hence, Gronwall lemma clearly implies \eqref{eq:975}. 

Regarding the time increments, let $s,t\in [0,T]$ with $s<t$ and $x\in \R^d$, and set $h:=t-s$. Then, using similar arguments and taking into account estimates \eqref{eq:301} and \eqref{eq:9966}, we have
\begin{align*}
\E{|u^n(s+h,x) - u^n(s,x)|^p} & \leq C_n h^{p\frac \beta 2} + C
\left(\int_s^{s+h}\int_{\R^d} G_{s+h-r}(x-y)dydr\right)^p \\
& \qquad \quad + C \int_0^s \sup_{y\in \R^d} \E{|u^n(r+h,y) - u^n(r,y)|^p}dr\\
& \leq  C_n h^{p\frac \beta 2} + h^p +  C \int_0^s \sup_{y\in \R^d} \E{|u^n(r+h,y) - u^n(r,y)|^p}dr\\
& \leq  C_n h^{p\frac \beta 2} + C \int_0^s \sup_{y\in \R^d} \E{|u^n(r+h,y) - u^n(r,y)|^p}dr.
\end{align*}	
Applying again Gronwall lemma, we get \eqref{eq:976} and therefore we conclude the proof. 	
\end{proof}


\section{Quasi-linear case: weak convergence}
\label{sec: proof-main}

This section is devoted to prove the main result of the paper, namely Theorem \ref{thm: main}.
Recall that $u^n=\{u^n(t,x),\, (t,x)\in [0,T]\times \R\}$ denotes the mild solution to \eqref{eq:
	waven} (resp. \eqref{eq: heatn}), which satisfies, for all $(t,x)\in[0,T]\times \R^d$,
\begin{equation}
	u^n(t,x)  = I^d_0(t,x)+ \int_{0}^{t}\int_{\R^d} G_{t-s}(x-y) W^n(ds,dy)  
	+ \int_{0}^{t} \big(b(u^n)*G_{t-s}\big)(x) ds, 
	\label{eq:500}
\end{equation}
where $G$ is the corresponding fundamental solution, $I_0^d$ is given by \eqref{eq:15} and $b$ is globally Lipschitz.  

Before getting involved in the proof, we have to make sure that the limit candidate $u$, defined as the solution to \eqref{eq: mild_u}, takes its values in the space $\C([0,T]\times \R^d)$. The following result addresses this issue.

\begin{proposition}\label{prop: uhol}
	Let $u$ be the solution of equation \eqref{eq: mild_u}, where $G$ is the fundamental solution of the wave equation (resp. heat equation) and $b$ is a Lipschitz function. Assume that the spectral measure $\mu$ satisfies, for some $q\in (0,2)$,
	\[
	\int_{\R^d}\frac{\mu(d\xi)}{1+|\xi|^q}<\infty.
	\]
	Consider the following assumptions on the initial data:
	\begin{enumerate}
		\item[(a)] Wave equation: (ii) in Hypothesis \ref{hyp: ic}.
		\item[(b)] Heat equation with bounded drift: (i) in Hypothesis \ref{hyp: ic}.
		\item[(c)] Heat equation with general drift: (i) in Hypothesis \ref{hyp: ic} and $u_0\in \C^\alpha(\R^d)$, for some $\alpha\in (0,1)$.
	\end{enumerate}
Then, the random field $u$ admits a version with (Hölder-)continuous paths.  
\end{proposition}

\begin{proof}
	In the cases (a) and (b), the proof can be built exactly in the same way as it has been done for Theorems \ref{thm: cont} and \ref{thm: cont-heat-bdd}, respectively. In the case (c), it is readily checked that we just need to follow the same steps as those in the proof of Proposition \ref{prop: heat}. 	
\end{proof}

The validity of Theorem \ref{thm: main} for the wave equation and for the heat equation with bounded drift is an immediate consequence of the results in sections \ref{sec: conv-linear}, \ref{sec: wave} and \ref{sec: heat-bounded}. More precisely, owing to steps 2 and 3 in the proof of Theorem \ref{thm: cont} (see, respectively, the final part of the proof of Theorem \ref{thm: cont-heat-bdd} for the case of the heat equation with bounded drift), we can infer that 
\beq
u^n=\big(F\circ T_{I_0^d}\big)(v^n),
\label{eq:400}
\eeq   
where $F$ is the operator defined in \eqref{eq:190}, which we proved to be a continuous functional, and $T_{I_0^d}:\C([0,T]\times \R^d)\to \C([0,T]\times \R^d)$ is the following translation operator:
\[
T_{I_0^d}(\eta)(t,x):=\eta(t,x)+I_0^d(t,x), \quad \eta\in \C([0,T]\times \R^d).
\]
Since $I_0^d$ is a continuous function (by \cite[Lem. 4.2]{Dalang-Quer}), $T_{I_0^d}$ is a well-defined continuous functional. In \eqref{eq:400}, we recall that $v^n$ denotes the stochastic convolution (see \eqref{eq:120}). In Section \ref{sec: conv-linear}, we showed that $v^n$ converges in law  to $v$, in the space $\C([0,T]\times \R^d)$, where $v$ is given by 
\[
v(t,x)=\int_0^t\int_{\R^d} G_{t-s}(x-y)W(ds,dy), 
\]
and $W$ is a Gaussian spatially homogeneous noise with spectral measure $\mu$ (see Hypothesis {\bf (H2)}). Hence, since $F\circ T_{I_0^d}$ defines a continuous operator on $\C([0,T]\times \R^d)$, the so-called Mapping theorem (see, e.g., \cite[Thm. 2.7]{Bill}) implies that $u^n$ converges in law to $u$, the solution of \eqref{eq: mild_u}. This concludes the proof of Theorem \ref{thm: main}  in the case of the wave equation and the case of the heat equation with bounded drift.   

From now on, we focus on the heat equation \eqref{eq: heatn} with a general globally Lipschitz drift $b$. In this case, in order to prove Theorem \ref{thm: main} we will follow a different strategy. Namely, first we check that the family of laws of $\{u^n,\, n\geq 1\}$ is tight in the space $\C([0,T]\times \R^d)$. Next, we will use Prohorov's theorem (see, e.g., \cite[Thm. 5.1]{Bill} and the Corollary that follows) in order to identify the limit law. 

\begin{proposition}\label{prop: tightness-heat}
Let $u^n$ be the solution of \eqref{eq: heatn}, which satisfies equation \eqref{eq:500} where $G$ is the heat kernel given by \eqref{eq: heat kernel}. We assume that $b$ is globally Lipschitz and  $u_0$ is measurable, bounded and $\alpha$-Hölder continuous for some $\alpha\in (0,1)$. Suppose that Hypothesis {\bf (H1)} holds. Then, the laws of $\{u^n,\, n\geq 1\}$ form a tight family in   	
$\C([0,T]\times \R^d)$.	
\end{proposition}

\begin{proof}
	The following statement is an immediate consequence of Proposition \ref{prop:tightness}: for all $p\geq 1$ and  $K\subset\R^d$ compact, there exists a constant $C>0$ such that, for all $x,z\in K$, 
	\beq
	\sup_{n\geq 1}\sup_{t\in [0,T]}
	\E{|v^n(t,x)-v^n(t,z)|^p} \leq C |x-z|^{p(1-\frac q 2)}, 
	\label{eq:600}
	\eeq
	and for all $s,t\in [0,T]$, 
	\beq
	\sup_{n\geq 1}\sup_{x\in \R^d}
	\E{|v^n(t,x)-v^n(s,x)|^p} \leq C |x-z|^{p(\frac12-\frac q 4)}. 
	\label{eq:601}
	\eeq
	Here, the parameter $q\in (0,2)$ is the one given in Hypothesis {\bf (H1)}. 
	
	Next, we repeat the proof of Proposition \ref{prop: heat} but using estimates \eqref{eq:600} and \eqref{eq:601} instead of 
	\eqref{eq:9955} and \eqref{eq:9966}, respectively. Thus, setting $h:=z-x$, we obtain that 
\[
\E{|u^n(t,x+h) - u^n(t,x)|^p} \leq
C |h|^{\beta p} 
+ C \int_0^t  \sup_{y\in \R^d}\E{|u^n(s,y+h) - u^n(s,y)|^p} ds,
\]
where $\beta=\min(\alpha,1-\frac q2)$. Note that now the constant appearing on the right-hand side above does not depend on $n$. Gronwall lemma let us conclude that 
\[
\sup_{n\geq 1}\sup_{t\in [0,T]} \E{|u^n(t,z) - u^n(t,x)|^p} \leq C |x-z|^{\beta p}
\] 
Regarding the time increments of $u^n$, we will end up with the estimate
\[
\sup_{n\geq 1}\sup_{x\in \R^d} \E{|u^n(t,x) - u^n(s,x)|^p} \leq C |t-s|^{\frac \beta 2 p}.
\] 
Therefore, it holds that
\[
\sup_{n\geq 1} \E{|u^n(t,x) - u^n(s,z)|^p}\leq C
\big(|t-s|+|x-z|\big)^{\frac \beta 2 p},
\]
for all $s,t\in [0,T]$ and $x,z\in K$. Taking $p$ sufficiently large, we can apply Theorem \ref{thm: criterion} and so conclude the proof. 
\end{proof}

The validity of Theorem \ref{thm: main} for the case of the heat equation with arbitrary Lipschitz drift is a consequence of Proposition \ref{prop: tightness-heat} and the next result.

\begin{proposition}\label{prop: iden}
	Let $u^n$ be the solution of \eqref{eq: heatn}, which satisfies equation \eqref{eq:500} where $G$ is the heat kernel given by \eqref{eq: heat kernel}. We assume that $b$ is globally Lipschitz and  $u_0$ is measurable, bounded and $\alpha$-Hölder continuous for some $\alpha\in (0,1)$. Suppose that Hypothesis {\bf (H2)} holds. Then, the finite-dimensional distributions of $u^n$ converge to those of $u$, as $n\to\infty$, where $u$ is the solution to \eqref{eq: mild_u}.
\end{proposition}

\begin{proof}
	First, we truncate the drift $b$ as follows. Let $m\geq 1$ and define
	$$b_m(x):=\begin{cases}  b(x)\wedge m, & \text{if }b(x)\geq0, \\
		b(x )\vee -m, & \text{if }b(x)<0.
	\end{cases}$$
	Then, the function $b_m$ is bounded and Lipschitz continuous, and converges
	pointwise to $b$, as $m\rightarrow \infty$. Moreover, a unique
	Lipschitz constant can be fixed for all functions $b_m$, $m\geq 1$,
	and $b$. Let $u_m^n$ be the solution of \eqref{eq: mild} with $b$ replaced by $b_m$.
	An immediate consequence of (b) in Theorem \ref{thm: main} is that, for any fixed $m\geq 1$,
	\begin{equation}
		u_m^n\xrightarrow[n\to\infty]{\mathcal{L}} u_m
		\label{eq:245}
	\end{equation}
	in the space $\C([0,T]\times \R^d)$, where  $u_m$ denotes the solution of \eqref{eq: mild_u} with $b$ replaced by $b_m$. Next, we claim that the following convergence is fulfilled:
	 \beq
	 \sup_{n\geq 1}\sup_{(t,x)\in [0,T]\times\R^d} \E{|u_m^n(t,x)-u^n(t,x)|^2}\xrightarrow[m\to \infty]{}0.
	 \label{eq:412}
	 \eeq
	 The proof of the above convergence follows exactly in the same way as in Step 2 of \cite[Sec. 4.3]{GJQ-Bernoulli}. The only needed auxiliary result is that, for all $p\geq 2$, it holds: 
	 \[
	 \sup_{n\geq 1}\sup_{(t,x)\in [0,T]\times\R^d} \E{|u^n(t,x)|^p}<\infty. 
	 \]
	 This estimate has been proved in Lemma \ref{lem: unif-moment}. Using the same arguments, one also shows that
	 \beq
	 \sup_{(t,x)\in [0,T]\times\R^d} \E{|u_m(t,x)-u(t,x)|^2}\xrightarrow[m\to \infty]{}0.
	 \label{eq:413}
	 \eeq
	 	 
	 At this point, we have all the ingredients to show that the finite-dimensional distributions of $u^n$ converge to those of $u$. The proof is similar to that of Step 3 in \cite[Sec. 4.3]{GJQ-Bernoulli}. We will give it for the sake of completeness. Let   $(t_1,x_1),\dots,(t_k,x_k)\in [0,T]\times \R^d$ and $f:\R^k\to\R$ be continuous and bounded. Then, we write
	 \begin{equation*}
	 	\begin{split}
	 		& \Big|\E{f\big(u^n(t_1,x_1), \dots,u^n(t_k,x_k)\big)-f\big(u(t_1,x_1),\dots,u(t_k,x_k)\big)}\Big| \\
	 		& \qquad \leq  \Big| \E{f\big(u^n(t_1,x_1),\dots,u^n(t_k,x_k)\big)-
	 			f\big(u_m^n(t_1,x_1),\dots,u_m^n(t_k,x_k)\big)}  \Big| \\
	 		& \qquad \quad +\Big| \E{f\big(u_m^n(t_1,x_1),\dots,u_m^n(t_k,x_k)\big)-
	 			f\big(u_m(t_1,x_1),\dots,u_m(t_k,x_k)\big)}  \Big| \\
	 		&  \qquad \quad  +\Big| \E{f\big(u_m(t_1,x_1),\dots,u_m(t_k,x_k)\big)-
	 			f\big(u(t_1,x_1),\dots,u(t_k,x_k)\big)}  \Big| \\
	 		& \qquad =: I_1(m,n)+I_2(m,n)+I_3(m).
	 	\end{split}
	 \end{equation*}
	 Without loosing any generality, we may assume that $f$ is Lipschitz continuous. Hence, we can argue as follows:
	 \begin{equation*}
	 	\begin{split}
	 		& \sup_{n\geq 1}  \Big| \E{f(u^n(t_1,x_1),\dots,u^n(t_k,x_k))-f(u_m^n(t_1,x_1),\dots,u_m^n(t_k,x_k))}  \Big| \\
	 		& \qquad \leq C \sup_{n\geq 1} \E{\Big(  \sum_{j=1}^{k} |u^n_m(t_j,x_j)-u^n(t_j,x_j)|^2  \Big)^{1/2}} \\
	 		& \qquad \leq  C \sup_{n\geq 1} \Big( \sum_{j=1}^{k} \E{|u^n_m(t_j,x_j)-u^n(t_j,x_j)|^2 }\Big)^{1/2} \\
	 		& \qquad \leq C k^{\frac12} \Big( \sup_{n\geq 1} \sup_{(t,x)\in [0,T]\times \R^d}
	 		\E{|u^n_m(t,x)-u^n(t,x)|^2 }\Big)^{1/2}.
	 	\end{split}
	 \end{equation*}
 Note that the latter term converges to $0$ as $m\to \infty$, by \eqref{eq:412}.
Thus, also taking into account \eqref{eq:413}, for any $\varepsilon>0$, there exists $m_0\geq 1$ such that, for all $m\geq m_0$, we have
	 $$ \sup_{n\geq 1} \Big(I_1(m,n)+I_3(m)\Big) \leq \frac{\varepsilon}{2}.$$
	 In particular, we have
	 \begin{equation*}
	 	\begin{split}
	 		\Big|\E{f(u^n(t_1,x_1), \dots,u^n(t_k,x_k))-f(u(t_1,x_1),\dots,u(t_k,x_k))}\Big|
	 		\leq  I_2(m_0,n) +\frac{\varepsilon}{2}.
	 	\end{split}
	 \end{equation*}
	  Finally, we observe that the convergence in law
	 \eqref{eq:245} implies the corresponding convergence of the finite dimensional distributions. Therefore, for some $n_0\geq 1$, we have, for all $n\geq n_0$,
	 $I_2(m_0,n)<\frac{\varepsilon}{2}$.
	 Hence,
	 $$\Big|\E{f(u^n(t_1,x_1), \dots,u^n(t_k,x_k))-f(u(t_1,x_1),\dots,u(t_k,x_k))}\Big|<\varepsilon.$$
	 Since $\varepsilon$ can be taken arbitrary small, we can conclude the proof.
\end{proof}


\appendix

\section{Tightness criterion}
\label{sec: app}

In the paper, we have made use of the following tightness criterion several times. 
Although this result seems to be well-known, we have not been able to find a proof in the literature, so we will give it for the sake of completeness.

\begin{theorem}\label{thm: criterion}
	Let $\{X_\lambda\}_{\lambda\in\Lambda}$ be a family of random variables in $\mathcal C(R)$, where $R$ is a closed rectangle of $\mathbb R^m$ that contains the origin. Then, the family of their laws is tight if the following conditions are fulfilled:
	\begin{itemize}
		\item[(a)] The laws of $\{X_\lambda(0)\}_{\lambda\in\Lambda}$ form a tight family.
		\item[(b)] There exist constants $C>0$, $\gamma\ge 1$ and  $\alpha>m$ such that, for all
		$x,y\in \R^m$,
		$$\sup_{\lambda\in \Lambda}\E{|X_\lambda(x)-X_\lambda(y)|^\gamma}\le C\,|x-y|^\alpha.$$
	\end{itemize}
\end{theorem}
\begin{remark} 
	If we have a family of random variables $\{X_\lambda\}_{\lambda\in\Lambda}$ in $\mathcal C([0,T]\times \mathbb R^d)$, for some $T>0$ and $d\in\mathbb{N}$, endowed with the topology of the uniform convergence on compact sets, the family of their laws is tight if the above conditions are satisfied for any closed rectangle $R\subset [0,T]\times \mathbb R^d$ with $0\in R$.
\end{remark}

In order to prove Theorem \ref{thm: criterion}, we will use the following result, which is a direct extension of \cite[Thm. 7.3]{Bill} to our setting.
\begin{theorem}\label{billingsley}
	Let $\{X_\lambda\}_{\lambda\in\Lambda}$ be a family of random variables in $\mathcal C(R)$, where $R$ is a closed rectangle of $\mathbb R^m$ that contains the origin. Then, the family of their laws is tight if and only if the following conditions are satisfied:
	\begin{itemize}
		\item[(i)] The laws of $\{X_\lambda(0)\}_{\lambda\in\Lambda}$ form a tight family.
		\item[(ii)] For any $\varepsilon>0$ and $\eta>0$, there exists  $\delta\in(0,1)$ such that,
		for any $\lambda\in \Lambda$,		
		$$\mathbb{P}\Bigg\{\sup_{\underset{|x-y|<\delta}{x,y\in R}}
		|X_\lambda(x)-X_\lambda(y)|\ge\varepsilon\Bigg\}\le \eta.$$
	\end{itemize}
\end{theorem}

We will also borrow the following version of the well-known Lemma of Garsia-Rodemich-Rumsey for metric spaces (see Appendix A in \cite{Dalang-Khosh-Eulalia}).

\begin{theorem}\label{eulalia}
	Let $\psi:\mathbb R\rightarrow \mathbb R_+$ be a function which is convex, even, strictly increasing in $\R_+$ and such that $\psi(0)=0$ and $\psi(\infty)=\infty$. Let $p:[0,\infty)\rightarrow\mathbb R_+$ be continuous, strictly increasing and such that $p(0)=0$.
	
	Let $(S,\varrho)$ be a metric space and $\nu$ a Radon measure on $S$. 	
	If $f:S\rightarrow \mathbb R$ is a continuous function, define 
	$$\Gamma=\int_S\int_S \psi\left(\frac{f(x)-f(y)}{p(\varrho(x,y))}\right)\nu(dx)\nu(dy).
	$$
	Let also $B_\varrho(x,r)$ be the open ball with center $x\in S$ and radius $r$. Then, if $\Gamma$ is a finite constant, it holds, for any $s,t\in S$:
	\begin{equation*}
		|f(x)-f(y)|
		\le  4\int_0^{2\varrho(x,y)}\left[\psi^{-1}\left(\frac{\Gamma}{[\mu(B_\varrho(x,\frac{u}2))]^2}\right)+\psi^{-1}\left(\frac{\Gamma}{[\mu(B_\varrho(y,\frac{u}2))]^2}\right)\right]p(du).
	\end{equation*}
\end{theorem}

\begin{remark}\label{rmk:13}
	Define, for any $u>0$,
	$$g(u):=\inf_{r\in S}\mu\big(B_\varrho(u/2,r)\big),$$ 
	and we assume that the above infimum  is strictly positive. Then, under the hypotheses of Theorem \ref{eulalia} and taking into account that $\psi^{-1}$ is an increasing function, 
	we have, for all $x,y\in S$:
	$$|f(x)-f(y)|\le 8\int_0^{2\varrho(x,y)}\psi^{-1}\left(\frac{\Gamma}{g(u)^2}\right)p(du).$$
\end{remark}

\medskip

In Remark \ref{rmk:13}, we take $S= R$, where $R$ is a closed rectangle of $\mathbb R^m$ that contains the origin, $\varrho$ the euclidean distance and $\nu$ the Lebesgue measure. Then,
$g(u)=C_m u^m$. 
Moreover, if we define 
$$\Gamma:=\int_R\int_R \psi\left(\frac{f(x)-f(y)}{p(|x-y|)}\right)dxdy,$$
and we assume that $\Gamma<\infty$, Remark \ref{rmk:13} implies that, for all $x,y\in R$:
\begin{equation}\label{GRR}
	|f(x)-f(y)|\le 8\int_0^{2|x-y|}\psi^{-1}\left(\frac{\Gamma}{C_m^2\,u^{2m}}\right)p(du).
\end{equation}
\medskip

With all these ingredients at hand, we can tackle the proof of Theorem \ref{thm: criterion}.

\medskip

\noindent {\it Proof of Theorem \ref{thm: criterion}}. 
We only need to show that condition $(b)$ of  Theorem \ref{thm: criterion} implies the validity of $(ii)$ in Theorem \ref{billingsley}. We take
$\psi(x)=|x|^{\gamma}$ and $p(x)=|x|^{\frac{k+2m}{\gamma}}$, with $k\in(0,\alpha-m)$.
Then, we have
\begin{align}\label{Gamma}
		\mathbb E\left[\int_R\int_R   \psi\left(  \frac{X_\lambda (x)-X_\lambda(y)}{p(x-y)}    \right)dxdy\right]& 
	=\int_R\int_R\mathbb E\left[\frac{|X_\lambda(x)-X_\lambda(y)|^{\gamma}}   {|x-y|^{k+2m}}        \right]dxdy \nonumber\\
	& = \int_R\int_R \mathbb E\left[\frac{|X_\lambda(x)-X_\lambda(y)|^{\gamma}}   {|x-y|^{\alpha}} \right]\frac1{|x-y|^{k+2m-\alpha}}dxdy \nonumber \\
	& \le C\int_R\int_R \frac1{|x-y|^{k+2m-\alpha}}dxdy \le M,
\end{align}
for some constant $M$, where we have applied condition (b) of Theorem \ref{thm: criterion} and the fact that  
$$\int_B\int_B \frac1{|x-y|^\beta}dxdy<\infty,$$
for any ball $B\subset \mathbb R^m$ with center in $0$ and for all $\beta<m$.

The estimate (\ref{Gamma}) implies that the random variables $\Gamma_{\lambda}$ defined as 
$$\Gamma_\lambda=\int_R\int_R \psi\left( \frac{X_\lambda(x)-X_\lambda (y)}{p(|x-y|)}\right)dxdy, \quad \lambda\in \Lambda,$$
are almost surely finite and that their expectation is bounded by $M$.
By (\ref{GRR}), we obtain that, for any $x,y\in R$,
$$|X_\lambda(x)-X_\lambda(y)|\le C\int_0^{2|x-y|}\frac{\Gamma_\lambda^{1/\gamma}}{u^{2m/\gamma}}\, u^{\frac{k+2m}{\gamma}-1}\, du = C |x-y|^{k/\gamma}\,
\Gamma_\lambda^{1/\gamma},
$$
and this implies that, for any $ \delta\in (0,1)$,
$$\sup_{\overset{x,y\in R}{|x-y|<\delta}}|X_\lambda(x)-X_\lambda(y)|\le C \delta^{k/\gamma}\Gamma_\lambda^{1/\gamma}.$$
Finally, we can check that condition (ii) of Theorem \ref{billingsley} is satisfied. Indeed, fix $\varepsilon>0$ and $\eta>0$ and apply Chebyshev's inequality:
\begin{align*}
\mathbb{P}\Bigg\{\sup_{\overset{x,y\in R}{|x-y|<\delta}}|X_\lambda(x)-X_\lambda(y)|\ge\varepsilon \Bigg\} & \le
\frac{\E{\sup_{\overset{x,y\in R}{|x-y|<\delta}}|X_\lambda(x)-X_\lambda(y)|^\gamma} }{\varepsilon^{\gamma}}\\
& \le C\frac{\delta^{k}\E{\Gamma_\lambda}}{\varepsilon^\gamma}\\
& \le \frac{CM\delta^{k}}{\varepsilon^{\gamma}}.
\end{align*}
The latter quantity can be made less than or equal to $\eta$ if $\delta$ is small enough.
\qed


\section*{Acknowledgement}

Research supported by the grant PID2021-123733NB-I00 (Ministerio de Economía y Competitividad, Spain).



\begin{thebibliography}{1}
	
\bibitem{Balan}R.M. Balan and X. Liang. Continuity in law for solutions of SPDEs with space-time homogeneous Gaussian noise.  Stochastics and Dynamics 23 (2023), no. 6.  	
	
\bibitem{Bezdek} P. Bezdek. On weak convergence of stochastic heat equation with colored noise.
Stochastic Processes and their Applications 126 (2016) 2860–2875.	
	
\bibitem{Bill} P. Billingsley. Convergence of Probability Measures, 2nd ed. Wiley Series in Probability and Statistics: Probability and Statistics. New York: Wiley, 1999. 	

\bibitem{CD} D. Conus and R.C. Dalang. The non-linear stochastic wave equation in high
dimensions. Electronic Journal of Probability 13 (2008).
	
\bibitem{DalangEJP} R.C. Dalang. Extending the martingale measure
stochastic integral with applications to spatially homogeneous s.p.d.e's.
Electron. J. Probab. 4 (1999), 1-29.

\bibitem{Dalang-Khosh-Eulalia} R.C. Dalang, D. Khoshnevisan and E. Nualart. Hitting probabilities for systems of non-linear stochastic heat equations with additive noise. ALEA, Lat. Am. J. Probab. Math. Stat. 3 (2007) 231-271.

\bibitem{Dalang-Quer} R.C. Dalang and L. Quer-Sardanyons. Stochastic integrals for spde’s: A comparison. 
Expositiones Mathematicae 29 (2011) 67–109.

\bibitem{Evans} L. C. Evans. Partial Differential Equations. American Mathematical Society, Providence, RI, 1998.

\bibitem{Folland} G.B. Folland. Introduction to Partial Differential Equations, Princeton Univ. Press, 1976.

\bibitem{Gel} I.M. Gel'fand and G.E. Shilov. Generalized Functions. Vol. 1. Properties
and Operations. Academic Press, New York, 1964.

\bibitem{GJQ-Bernoulli} L.M. Giordano, M. Jolis and L. Quer-Sardanyons.
SPDEs with fractional noise in space:
Continuity in law with respect to the
Hurst index. 
Bernoulli 26 (2020), 352-386. 

\bibitem{GJQ-SPA} L.M. Giordano, M. Jolis and L. Quer-Sardanyons. SPDEs with linear multi-
plicative fractional noise: continuity in law with respect to the Hurst index. Stoch.
Proc. Their Appl. 130 (2020), 7396-7430.

\bibitem{MS} A. Millet and M. Sanz-Solé. A stochastic wave equation in two space dimension: Smoothness of the
law. Ann. Probab. 27 (1999), no. 2, 803-844.

\bibitem{NQ} D. Nualart and L. Quer-Sardanyons. Existence and smoothness of the density for spatially homogeneous spde’s. Potential Analysis 27 (2007), 281-299.

\bibitem{Samorod} G. Samorodnitski and M.S. Taqqu. Stable non-Gaussian random processes, Chapman \& Hall, London,1994.
Paris, 1966.

\bibitem{Sanz-Sarra} M. Sanz-Solé and M. Sarrà. Path properties of a class of Gaussian processes with applications to spde's. Canadian Mathematical Society, Conference Proceedings 28 (2000) 303-316.  

\bibitem{SS} M. Sanz-Solé and M. Sarrà. Hölder Continuity for the Stochastic Heat Equation With Spatially Correlated Noise. In: Dalang, R.C., Dozzi, M., Russo, F. (eds) Seminar on Stochastic Analysis, Random Fields and Applications III. Progress in Probability, vol 52. Birkhäuser, Basel, 2002.

\bibitem{Schwartz} L. Schwartz. Théorie des Distributions, Hermann,
Paris, 1966.

\bibitem{Tao} W. Tao. On weak convergence of stochastic wave equation with colored noise on $\R$. ArXiv Preprint  arXiv:2408.10326 



\end{thebibliography}
\end{document}